\documentclass[12pt]{amsart}
\usepackage{amsmath,amscd,amsthm,amsfonts, amssymb,amsxtra, mathrsfs, amsrefs}
\usepackage[us,12hr]{datetime}
\usepackage[all]{xy} \SelectTips{cm}{}
\usepackage{hyperref}
  \hypersetup{colorlinks=true,citecolor=blue}
   \usepackage{graphicx}
   \usepackage[font=scriptsize]{caption}

\CDat
\newtheorem{thm}{Theorem}[section]  
\newtheorem*{un-no-thm}{Theorem}
\newtheorem{cor}[thm]{Corollary}     
\newtheorem{lem}[thm]{Lemma}         
\newtheorem{prop}[thm]{Proposition}

\newtheorem{bigthm}{Theorem}
\newtheorem{bigcor}[bigthm]{Corollary}

\theoremstyle{definition}
\newtheorem{defn}[thm]{Definition}   

\theoremstyle{definition}

\theoremstyle{definition}
\theoremstyle{remark}
\newtheorem{rem}[thm]{Remark}

\newtheorem{notation}[thm]{Notation}
\newtheorem*{acks}{Acknowledgements}

\newtheorem*{intro-rem}{Remark}
\newtheorem*{intro-rems}{Remarks}

\newtheorem{ex}[thm]{Example}

\DeclareMathOperator{\TR}{TR}
\DeclareMathOperator{\Z}{\mathbb Z}

\begin{document}
\title{$K$-theoretic torsion and the zeta function}
\author{John R. Klein} 
\address{Wayne State University,
Detroit, MI 48202} 
\email{klein@math.wayne.edu} 
\author{Cary Malkiewich} 
\address{Binghamton University, Binghamton, NY 13902}
\email{malkiewich@math.binghamton.edu}
\begin{abstract} We generalize to higher algebraic $K$-theory
an identity (originally due to Milnor) that relates the Reidemeister
torsion of an infinite cyclic cover to its Lefschetz zeta function.
Our identity involves a higher torsion invariant, the {\it endomorphism torsion,} of a parametrized
family of endomorphisms as well as a {\it higher zeta function} of such a family.  
We also exhibit several examples of families of endomorphisms having non-trivial endomorphism torsion. 
\end{abstract}
\thanks{}
\maketitle
\setlength{\parindent}{15pt}
\setlength{\parskip}{1pt plus 0pt minus 1pt}
\def\smsh{\wedge}
\def\flush{\flushpar}
\def\dbslash{/\!\! /}
\def\:{\colon\!}
\def\Bbb{\mathbb}
\def\bold{\mathbf}
\def\cal{\mathcal}
\def\End{\text{\rm End}}
\def\Aut{\text{\rm Aut}}
\def\map{\text{\rm map}}
\def\sh{\text{\rm sh}}
\def\orb{\cal O}
\def\hoP{\text{\rm ho}P}
\def\Ch{\text{\rm Ch}}
\def\sperf{\text{\rm sperf}}
\def\perf{\text{\rm perf}}
\def\proj{\text{\rm proj}}
\def\fd{\text{\rm fd}}
\def\cfd{\text{\rm cfd}}
\def\gl{\text{\rm GL}}
\def\sm{\text{\rm sm}}

\setcounter{tocdepth}{1}
\tableofcontents
\addcontentsline{file}{sec_unit}{entry}

\section{Introduction} \label{sec:intro}
Algebraic $K$-theory is a natural setting for algebraic torsion invariants such Reidemeister and Whitehead torsion. In this paper we consider a closely related invariant that we call \emph{endomorphism torsion}. Endomorphism torsion provides a means of conceptually relating Reidemeister torsion to dynamical systems. Specifically, it is related to the ``counting'' of periodic orbits in a dynamical system by work of Milnor  \cite{Milnor_inf_cyclic}, and subsequently
partially developed in the papers of Fried \cite{Fried2}, \cite{Fried1} and  Geoghegan and Nicas \cite{Geoghegan-Nicas1}.
 
Milnor considers infinite cyclic coverings $\widetilde X \to X$ with generating
 deck transformation $T\: \widetilde X \to \widetilde X$. Assume that $\widetilde X$
has finite rational homology. One can then consider two kinds of invariants associated with
$(\widetilde X,T)$.

The first invariant associated with $(\widetilde X,T)$ is the algebraic torsion $\tau(t) \in \Bbb Q(t)^\times$ of the contractible chain complex
$C_*(\widetilde X) \otimes_{\Bbb Z[t,t^{-1}]} \Bbb Q(t)$. This is the Reidemeister torsion of $X$ with respect to the homomorphism $\pi_1(X) \to \Bbb Z \to \Bbb Q(t)^\times$ coming from the cyclic cover $\widetilde X \to X$ and the map $n \mapsto t^n$. Normally $\tau(t)$ would only be defined up to an element of $\Bbb \Z[t,t^{-1}]^\times = \{\pm t^n\}_{n\in \Bbb Z}$, but in this paper we re-interpret $\tau(t)$ as the endomorphism torsion of the complex $C_*(\widetilde X)$. As a result, it is an element of $\Bbb Q(t)^\times$ with no indeterminacy.


The other invariant associated with $(\widetilde X,T)$
is the Lefschetz zeta function
$$
\zeta(t) := \text{exp}\left(\sum_{k \ge 1} L(T^{\circ k}) \frac{t^k}{k}\right) \in \Bbb Q(t)^\times\, ,
$$
which encodes counting the Lefschetz numbers of the iterates of $T$ taken collectively.
Here $L(T^{\circ k})$ is the Lefschetz number of the $k$-fold composite $T^{\circ k}$, which is in some sense a homological count of the $k$-fold periodic points of $T$.

Milnor's remarkable formula is the functional equation
\begin{equation} \label{eqn:torsion-zeta}
\zeta(t^{-1}) \tau(t) = t^{\chi(\widetilde X)}\, ,
\end{equation}
where $\chi(\widetilde X)$ is the Euler characteristic. 
In this paper we show that this equation \eqref{eqn:torsion-zeta} is a consequence of another equation that holds in the higher algebraic $K$-theory of endomorphisms. 

Let $\End_A$ denote the endomorphism category of an associative unital ring $A$. The objects of $\End_A$ are pairs $(P,f)$ 
such that $P$ is a finitely projective right $A$-module, and $f\: P \to P$ is an endomorphism. 
A morphism $(P_1,f_1) \to (P_2,f_2)$ is a homomorphism of $A$-modules $g\: P_1\to P_2$ such that $f_2g = gf_1$.

Grayson \cite{Grayson} considers a variant of the endomorphism
category that is based on a localizing parameter.
Recall that a polynomial $p\in A[t]$ is {\it centric} or {\it central} if it lies in $Z(A[t]) = Z(A)[t]$, where $Z(A)$ is the center of $A$. Let $S\subset A[t]$ be any multiplicative subset of monic centric polynomials.
Define the full subcategory of $S$-torsion endomorphisms
\[
\End^S_A \subset \End_A
\]
whose objects are those $(P,f)$ such that $p(f) = 0$ for some  $p\in S$. 
Let
\[
A[t]_S := S^{-1} A[t]
\]
denote the localization of $A[t]$ with respect to $S$.

\begin{ex} Suppose $A$ is commutative and $S\subset A[t]$ is the set of all monic 
polynomials. Then  $\End^S_A = \End_A$.
If $A$ is an integral domain, then $A[t]_S = A(t)$ is the ring of rational functions.
\end{ex}

\begin{ex} Let $S = \{t^n \ | \ n \ge 0\}$. Then 
$\End^S_A = \text{Nil}_A$ is the category of nilpotent endomorphisms,
and $A[t]_S = A[t,t^{-1}]$ is the ring of Laurent polynomials.
\end{ex}

 In \S\ref{sec:torsion}  we construct a map of $K$-theory spaces
 \[
 \tau(t)\:  K(\End^S_A) @> >> \Omega K(A[t]_S)
 \]
 called the {\it higher endomorphism torsion map.} Note that in contrast to more familiar kinds of torsion, this induces a map on all homotopy groups, not just $\pi_0$, and has no indeterminacy.
 
\begin{rem}
	Throughout this paper, we regard higher $K$-theory as a space rather than a spectrum. This is only for simplicity -- the arguments also work for spectra, we just have to include the iterations of the $S.$ construction and the compatibility between them.
\end{rem}
 
 For $p\in S$ of degree $d$, set $\tilde p(t) := p(1/t)t^d$.
Define $T\subset A[t]$ to be the set $\{\tilde p\ |\ p \in S\}$.
 Then $T$ is a multiplicative subset of centric polynomials with leading term 1.
 In $\S\ref{sec:zeta}$ we 
 define the {\it higher $K$-theory zeta function}
 \[
  \zeta(t)\:  K(\End^S_A) @> >> \Omega K(A[t]_T)\, ,
  \]
 following \cite{Lindenstrauss-McCarthy}. This zeta function is related to the more classical ones via the trace map to the topological de~Rham-Witt space $\text{TR}(A)$, see Remark \ref{sec:deRham-Witt} and the recent preprints \cite{Malkiewich-et-al,dknp}.
 
 
 Assume $t\in S$.  Then there is a canonical ring homomorphism
 \[
 A[t^{-1}]_T \to A[t]_S\, ,
 \]
that is induced by   $t^{-1} \mapsto 1/t$.  Therefore, if we substitute $t$ by $t^{-1}$,
the zeta function $\zeta(t^{-1})$ can also be regarded as a map into 
$\Omega K(A[t]_S)$.

Informally, our main result is the following statement.

\begin{bigthm}\label{bigthm:milnor} The loop product of the maps
\[
\zeta(t^{-1}), \tau(t)\: K(\End^S_A) @> >> \Omega K(A[t]_S)
\]
is homotopic to the map that sends each endomorphism $(P,f)$ to 
the loop defined by multiplication by $t\: P[t]_S\to P[t]_S$, where
$P[t]_S$ is the module $P \otimes_A A[t]_S$.
\end{bigthm}
\noindent 
Taking $S\subset A[t]$ to be all monic centric polynomials and passing to $\pi_0$ recovers Milnor's formula \eqref{eqn:torsion-zeta}. For a more precise statement, see Theorem \ref{thm:milnor} below. 
\medskip

The forgetful functor $(P,f) \mapsto P$ induces a map $K(\End^S_A) \to K(A)$. Let $\widetilde K(\End^S_A)$ be its 
 homotopy fiber; then $\widetilde K(\End^S_A)\to K(\End^S_A)$ is homotopically a retract.  
We show that  $t\: K(\End^S_A) @> >> \Omega K(A[t]_S) $ is homotopically trivial when restricted to
$\widetilde K(\End^S_A)$. Let 
\[
\tilde \tau(t),\tilde \zeta(t^{-1})\: \widetilde K(\End^S_A) @> >> \Omega K(A[t]_S) 
\] denote the respective restrictions of $\tau(t)$ and $\zeta(t^{-1})$
to $\widetilde K(\End^S_A)$.

\begin{bigcor} As homotopy classes of maps, $\tilde \tau(t)$ and $\tilde \zeta(t^{-1})$ are additive inverses. 
In particular, the homotopy class of $\tilde \tau(t)$ admits
 a canonical factorization through $\Omega K(A[t^{-1}]_T)$.
\end{bigcor}

In \S\ref{sec:families}
we construct higher endomorphism torsion and zeta function invariants for families of endomorphisms of spaces, i.e., data of the form
\[
(p\: E\to B,f\: E\to E)\, ,
\]
in which $p$ is fibration with finitely dominated fibers, and $f$ is a self-map covering the identity map of $B$.
Hence, we may regard $E$ as fiberwise space over $B$ equipped with fiberwise $\Bbb N$-action. 
Both the higher endomorphism torsion $\tau(p,f)$ and zeta function $\zeta(p,f)$ are invariant with respect to
$\Bbb N$-equivariant fiberwise weak equivalence.

For any preselected  ring homomorphism $\Bbb Z[\pi_1(E)] \to A$ and subset $S \subset A[t]$ as above, the torsion invariant $\tau(p,f)$ is a composition of the form
\[
B @> c >> K(\End^S_A)@> \tau(t)>>  \Omega K(A[t]_S)\, ,
\]
where the map $c$ is induced by the classifying space data for the family of endomorphisms.

In particular, taking $B = S^n$ to be the $n$-sphere, the torsion produces an invariant in $K_{n+1}(A[t]_S)$. In Theorem \ref{thm:non-trivial-example} we prove the following:

\begin{bigthm}  \label{bigthm:example}  There is a pair
\[
(p\: E\to S^1, f\: E\to E)
\]
for which the component of $\tau(p,f)$ in $K_2(\Bbb Q(t))$ is non-trivial.
\end{bigthm}

In the above, the fibration $p$ has fiber $S^n \vee S^n$ and the monodromy $\theta\: S^n \vee S^n \to S^n \vee S^n$ has degree given by the matrix
\[
Q:= \begin{bmatrix}
0& - 1\\
1 & 1
\end{bmatrix} \, .
\]
As $I-Q$ is invertible over the integers, $E$ is a non-trivial homology circle. 
The self-map $f\: E\to E$ is the map of mapping tori induced by $\theta$. 

\begin{rem} The example in Theorem \ref{bigthm:example} is explicit. In \S\ref{subsec:examples}
we provide several other explicit examples having non-trivial torsion in $K_2(\Bbb Q(t))$.

In \S\ref{subsec:higher} we employ a kind of reverse engineering to exhibit higher
dimensional examples with non-trivial torsion in $K_{2n+2}(\Bbb Q(t))$ for $n > 1$.  The higher dimensional
examples are fibrations over homology $(2n+1)$-spheres equipped with a fiberwise self-map.
However, these examples are not explicit and their existence relies on recent work by Andrew Salch \cite{Salch}.
\end{rem}

\subsection{Further results} 

The paper also contains three appendices. 
\medskip

Appendix \ref{sec:fund-thm} recalls two ``fundamental theorems'' of 
endomorphism $K$-theory. The first
relates $K(\End^S_A)$ to $\Omega K(A[t]_S)$ and can be derived
without much effort from Waldhausen's generic fibration theorem. We sketch a proof since we could not find it stated in the literature. 
The other result, which is due to Grayson and is better known, relates the 
 reduced $K$-theory space $\widetilde K(\End^S_A)$
to $\Omega K(A[t^{-1}]_T)$. 
We make little claim to originality except to note that
the zeta function provides Grayson's equivalence and the endomorphism torsion provides the other one.

In Appendix  \ref{sec:boundary} we give an explicit description in terms of Waldhausen's $S.$ construction of the boundary map $\partial$ of the $K$-theory localization sequence associated with the localization $A[t] \to A[t]_S$.
 In this case there is a short exact sequence
\[
0 \to K_{n+1}(A[t]) \to K_{n+1}(A[t]_S) @> \partial >> K_n(\End^S_A) \to 0
\]
that is split by the endomorphism torsion map. This result is used in the proof of Theorem \ref{bigthm:example}.

In Appendix \ref{sec:sphere_theorem} we verify the hypotheses of Waldhausen's ``sphere theorem'' \cite[thm.~1.71]{Wald_1126} for the categories of endomorphisms $\End_A$ and $\End_A^S$. This is somewhat non-trivial due to the discrepancy between the categories of perfect complexes over the rings $A[t]$ and $A$.


\begin{acks} The first author was partially supported by Simons Collaboration Grant 317496. The second author was partially supported by NSF DMS-2005524. 
\end{acks}

\section{Preliminaries}\label{sec:preliminaries}

\subsection{$K$-theory of Waldhausen categories}  
Let $C$ be a Waldhausen category, i.e., a category
 with cofibrations $\text{co}C$ and weak equivalences $wC$. It is customary to denote the
cofibrations by $\rightarrowtail$ and the weak equivalences by $@>\sim >>$. When the isomorphisms of $C$ are used as the weak equivalences, we denote $wC$ by $iC$.

The {\it $K$-theory} 
of $(C,\text{co} C,wC)$ is the based loop space
\[
\Omega |wS.C|\, ,
\]
where $|wS.C|$ denotes the realization of a simplicial pointed category $wS.C$. The objects of $wS_kC$ are given by filtered objects of $C$
of the form
\[
A_\bullet :=  A_1 \rightarrowtail A_2 \rightarrowtail \cdots \rightarrowtail A_k
\]
together with explicit choices of quotient data
\[
A_{ij} := A_j/A_i \in C\, \quad i\le j\, .
\]
See \cite{Wald_1126} or \cite{Weibel_K-book} for more details.
%

In particular, $wS_0C$ is the trivial category with object $\ast$, and $wS_1C = wC$. As a result there is a canonical map
\begin{equation} \label{eqn:1-skeleton}
|wC| @>>> \Omega |wS.C|
\end{equation}
which is adjoint to the inclusion of the one-skeleton in the simplicial direction.

\begin{ex}[Projective modules] Let $A$ be an associative unital ring. 
Let $P(A)$ be the category of finitely generated projective (right) $A$-modules.  A cofibration of $P(A)$
is a monomorphism $P \to Q$ whose cokernel is projective. A weak equivalence of $P(A)$ is an isomorphism. The category of weak equivalences is denoted $iP(A)$. Then the $K$-theory space of $A$ is
\[
K(A) := \Omega |iS.P(A)|.
\] 
This coincides up to equivalence with the Quillen $K$-theory of $P(A)$ considered as an exact category 
\cite[\S1.9]{Wald_1126}
\end{ex}

\begin{ex}[Endomorphisms of modules]
Let $\End_A = \End(P(A))$ be the category in which an object is a pair $(P,f)$, where
\begin{itemize}
	\item $P$ a finitely generated projective $A$-module, and
	\item $f\: P \to P$ is an $A$-linear endomorphism.
\end{itemize}
A morphism $(P,f) \to (Q,g)$ of $\End_A$ is a homomorphism $h\: P \to Q$ such that $gh = hf$.

We equip $\End_A$ with a Waldhausen category structure induced by the forgetful functor 
$\End_A \to P(A)$. That is, a morphism $h\:(P,f) \to (Q,g)$ is a cofibration/weak equivalence if and only if
$h\: P \to Q$ is cofibration/weak equivalence in $P(A)$.
The {\it $K$-theory of endomorphisms} is 
\[
K(\End_A) := \Omega |iS.\End_A|\, .
\]
\end{ex}

\begin{ex}[Endomorphisms with $S$-torsion]\label{ex:torsion_endos}
Each endomorphism $(P,f)$ is an $A[t]$-module with $t$ acting by $f$. Here $A[t]$ is the polynomial ring, the free $A$-algebra on one generator $t$ subject to the relation that $t$ commutes with $A$. Given a polynomial $p(t)$, we denote its action on $P$ by $p(f)$. Note that if $p(t)$ is in the center $Z(A[t]) = Z(A)[t]$, then $p(f)$ is an $A[t]$-linear map.

Let $S \subseteq A[t]$ be any multiplicatively closed set of monic centric polynomials. Then we define $\End_A^S = \End(P(A))^S$ be the full subcategory of $\End_A$ on the objects $(P,f)$ such that $p(f) = 0$ for some $p(t)\in S$. Equivalently, the localized module $S^{-1}P$ is zero.

When $A$ is commutative and $S$ consists of all monic polynomials, then $\End_A^S = \End_A$. In other words, every endomorphism in $\End_A$ is monic-torsion. This is because we could take $p(t)$ to be the characteristic polynomial of any extension of $f$ to a finite free module. This argument does not apply outside the commutative case because we cannot define determinants in the usual way.
\end{ex}

\begin{ex}[Automorphisms]
The full subcategory $\Aut_A\subset \End_A$ consisting of those objects $(P,f)$ such that $f$ is an automorphism
is called the {\it automorphism category} of $A$. The {\it $K$-theory of automorphisms} is
\[
K(\Aut_A) := \Omega|iS.\Aut_A|\, .
\]
\end{ex}
 
 \begin{ex}[Extension of scalars]  Let $\phi\colon A\to B$ be a ring homomorphism. Then $B$ has the structure of an $A$-$A$-bimodule.
 The operation
 \[
 (P,f) \mapsto (P\otimes_A B,f\otimes 1)
 \]
 is an exact functor. Consequently, it induces a map
 \[
 K(\End_A) \to K(\End_B)\, .
 \]
 \end{ex}

For the first segment of the paper these examples will be sufficient, but see \S\ref{sec:chain_complexes} for refinements of these examples to chain complexes.

\subsection{Categorical loops}
For a category $C$ let
 $\End(C)$ be the category  whose objects are self-maps
 $\phi_a\: a\to a$ of $C$ and whose morphisms $(a,\phi_a) \to (b,\phi_b)$ are maps
 $u\:a\to b$ which fit into a commutative diagram
 \[
 \xymatrix{
 a \ar[r]^u \ar[d]_{\phi_a} & b\ar[d]^{\phi_b}\\
 a\ar[r]_u & b\, .
 }
 \]
 There is a ``coassembly'' map of spaces
 \begin{equation} \label{eqn:loops}
 e\: |\End(C)| @>>> L|C|\, ,
 \end{equation}
 where  $L|C|$ is the unbased loop space of $|C|$. This map commutes with the projection of each space to $|C|$, by sending $(a,\phi_a)$ to $a$ or by evaluating a loop at the basepoint.
 
The coassembly map may be described in the following way: each $k$-simplex in the realization $ |\End(C)|$ is represented by a functor $f\:[1] \times [k] \to C$.
Here, $[k]$ is the category with objects $i$ for $0\le i\le k$ and a unique morphism from $i \to j$ whenever $i\le j$.
The functor $f$ also has the property that its restrictions to the two objects $0,1 \in [1]$ give the same functor $[k] \to C$. We take realization to get a map $(\Delta^1/\partial \Delta^1) \times \Delta^k @>>> |C|$ whose adjoint is a map $\Delta^k \to L|C|$. It is easy to check this rule respects faces and degeneracies, giving a well-defined map $|\End(C)| @>>> L|C|$. Furthermore, restricting to the basepoint of $S^1$ gives precisely the map $|\End(C)| @>>> |C|$ that forgets the endomorphism $\phi_a$.

Applying the above to $iS_k\Aut_A = \End(iS_k P(A))$, we obtain a map
\[ |iS_k\Aut_A| \to L|iS_kP(A)| \]
which becomes a map of simplicial spaces when $k$ is varied.
Taking the realization of the adjoint, we obtain a map
 $S^1 \times |iS.\Aut_A| \to |iS.P(A)|$, which after looping once defines a map
\[ S^1 \times K(\Aut_A) \to K(A)\, . \]
Taking the adjoint again defines a coassembly map on $K$-theory
 \begin{equation} \label{eqn:coassembly}
 c\colon K(\Aut_A) \to LK(A)\, .
 \end{equation}

\begin{rem}
Note that this $K$-theory coassembly map does not extend to $K(\End_A)$; it is necessary to use isomorphisms to actually get a loop in the target space. More generally, we can define such a coassembly map for any Waldhausen category $C$, and it takes the form
\[ K(\Aut_C) \to LK(C)\, , \]
where $\Aut_C$ is the category of objects $c \in C$ and maps $c @>\sim>> c$ in $wC$.
\end{rem}

\subsection{Composition}
Composition of endomorphisms defines a functor
\[ \End(P(A)) \times_{P(A)} \End(P(A)) @> \circ >> \End(P(A)). \]
It is straightforward to check that it is strictly associative and has unit given by the functor $P(A) \to \End_A$ taking each module to its identity morphism. Since $K$-theory commutes with pullbacks, this gives maps
\[ K(\End_A) \times_{K(A)} K(\End_A) @> \circ >> K(\End_A) \]
making the space $K(\End_A)$ into a fiberwise topological monoid over the space $K(A)$. Restricting further to automorphisms makes $K(\Aut_A)$ into a fiberwise topological group over $K(A)$.

\begin{rem}
	This composition operation does not induce a composition on the homotopy fiber $\widetilde K(\End_A)$ because $K(\End_A) \to K(A)$ is not a fibration. This would lead to a contradiction anyway: using the characteristic polynomial map $\widetilde K_0(\End(A)) \to (1 + tA[[t]])^\times$, we would get an incorrect identity $p(f \circ g) = p(f)p(g)$, where $p$ is the characteristic polynomial of an endomorphism $f$.
\end{rem}

Similarly, the free loop space $LK(A)$ is a fiberwise grouplike $A_\infty$-space over $K(A)$, with composition $*$ given by composing loops (using the little intervals operad). Note that $LK(A) \to K(A)$ is a fibration and so this does give a composition on the homotopy fiber, namely the usual composition on the based loop space $\Omega K(A)$.

\begin{lem}\label{coassembly_respects_composition}
	The following square commutes up to homotopy: 
	\[ \xymatrix{
		K(\Aut_A) \times_{K(A)} K(\Aut_A) \ar[d]_\circ \ar[r] & LK(A) \times_{K(A)} LK(A) \ar[d]^-{{*}} \\
		K(\Aut_A) \ar[r] & LK(A)\, .
	} \]
	Here, the horizontal maps are induced by the coassembly map \eqref{eqn:coassembly}.
\end{lem}


\begin{proof}
	We first describe a natural homotopy making the following square commute.
	\[ \xymatrix{
		|\End(C)| \times_{|C|} |\End(C)| \ar[d]_-\circ \ar[r] & L|C| \times_{|C|} L|C| \ar[d]^-{{*}} \\
		|\End(C)| \ar[r] & L|C|\, ,
	} \]
	where the horizontal maps are induced by the coassembly map \eqref{eqn:loops}.
	Since realization commutes with pullbacks, each $k$-simplex in the top-left is described as a pair of functors $[1] \times [k] \to C$ that restrict to the same functor $[k] \to C$ on the endpoints of $[1]$. These together specify a functor $[2] \times [k] \to C$, which realizes to a map
	\begin{equation}\label{composition_homotopy}
		\Delta^2 \times \Delta^k \to |C|.
	\end{equation}
	The left-bottom route of the commuting diagram is the restriction of \eqref{composition_homotopy} to the face $[0 < 2]$ of $\Delta^2$, while the top-right route is the restriction to the two faces $[0 < 1]$ and $[1 < 2]$, composed together using some fixed rule for composing two loops. The $\Delta^2$ provides a homotopy between them. So if we fix a choice for this homotopy without reference to $\Delta^k$, we get a homotopy that commutes with face and degeneracy maps and therefore gives a well-defined homotopy on the entire realization.
	
	Now we apply this to $C = iS_kP(A)$ for each $k \geq 0$ to get a homotopy that commutes the following maps of simplicial spaces:
	\[ 
	\xymatrix{
		|iS_k\Aut_A| \times_{|iS_kP(A)|} |iS_k\Aut_A| \ar[d]_\circ \ar[r] & L|iS_kP(A)| \times_{|iS_kP(A)|} L|iS_kP(A)| \ar[d]^-{{*}} \\
		|iS_k\Aut_A| \ar[r] & L|iS_kP(A)|\, ,
	} 
	\]
where the horizontal maps arise from \eqref{eqn:loops}.	
	Taking realizations, the left column becomes the left column in the statement of the proposition, while the right column admits a strictly commuting map to the right column in the statement of the proposition.
	Along these identifications, the realization of our previous homotopy becomes the desired homotopy.
\end{proof}

\section{Higher endomorphism torsion} \label{sec:torsion}

\subsection{Some extensions of the polynomial ring}\label{subsec:some-localizations}
Let $A$ be an associative unital ring, and let $A[t]$ be the ring of polynomials with coefficients in $A$. By convention, this means that the indeterminate $t$ commutes with $A$. A polynomial is centric if it lies in the center $Z(A[t]) = Z(A)[t]$.

Following Grayson \cite{Grayson}, fix a multiplicative subset $S\subset A[t]$ of monic centric polynomials in $A[t]$, containing at least the powers of $t$. Let
\[
A[t]_S := S^{-1} A[t]
\]
denote the localization of $A[t]$ given by the left fractions with respect to $S$.

For each polynomial
\[ p(t) = t^n + a_1t^{n-1} + a_2t^{n-2} + \ldots + a_{n-1}t + a_n \in S\, ,\hspace{2.2em}\mbox{} \]
let $\tilde p(t)$ denote the re-normalized polynomial
\[ \tilde p(t) = 1 + a_1t + a_2t^2 + \ldots + a_{n-1}t^{n-1} + a_nt^n = p(1/t)t^n\, . \]
This defines a second multiplicative subset $T = \{\tilde p(t)|\, p(t) \in S\}$ of centric polynomials having leading coefficient 1. 
Let
\[
A[t]_T := T^{-1}A[t]_T\, .
\]
The homomorphisms $A[t] \to A[t]_S$ and $A[t] \to A[t]_T$  are both flat localizations \cite{Weibel-Yao}.

 The universal property of localization
gives canonical, injective ring homomorphisms
\begin{equation}\label{eqn:universal-maps}
A[t]_T \to A[[t]] \qquad \text{and} \qquad  A[t^{-1}]_T \to A[t]_S \to A(t)\, ,
\end{equation}
where $A[[t]]$ is the ring of formal power series over $A$ and $A(t)$ is the localization of $A[t]$ at the set of all central non-zero-divisors. The first of these homomorphisms is induced by the inclusion $A[t]\to A[[t]]$. The second is
is induced by the homomorphism $A[t^{-1}] \to A[t]_S$  given by $p(t^{-1}) \mapsto (t^np(t^{-1}))/t^n$, where $n$ is
the degree of $p(t^{-1})$. If $A$ is a field then $A(t)$ is the field
of rational functions; in this case if $S$ consists of all monic polynomials, then the homomorphism $A[t]_S\to A(t)$ is an isomorphism.


For a finitely generated projective (right) $A$-module $P$ we define
\[
P[t] := P \otimes_A A[t] \cong \bigoplus_{t=0}^\infty P, \qquad P[[t]] := P \otimes_A A[[t]] \cong \prod_{t=0}^\infty P.
\]
Note that the second isomorphism holds because $P$ is a summand of a finite free module.

\subsection{Higher endomorphism torsion}
Recall from Example \ref{ex:torsion_endos} that $\End_A^S$ is the category of all endomorphisms of finitely generated projective $A$-modules that are $S$-torsion. If $(P,f)$ is an object of $\End_A^S$ then there is a short exact sequence of of $A[t]$-modules
 \begin{equation}\label{eq:char-sequence}
 0 \to P[t]@> t-f >> P[t] \to P \to 0\, ,
\end{equation} 
called the {\it characteristic sequence} of $P$. Here $t-f$ is shorthand for $1\otimes t - f\otimes 1$. The action of $t$ on the first two terms is by $\cdot t$, and on the third term is by $f$. Let  
	 \[
	 \tau_f(t):  P[t]_S @> t-f >\cong > P[t]_S
	 \]
	 denote the localization of the first map.
Note that $\tau_f(t)$ is an isomorphism, by the assumption that $P_S = S^{-1}P$ vanishes.
 
 The rule $(P,f) \mapsto (P[t]_S,\tau_f(t))$ defines an exact functor 
 \[
\End_{A}^S  @>>> \Aut_{A[t]_S} 
 \]
 and therefore a map of $K$-theory spaces $K(\End_A^S) \to K(\Aut_{A[t]_S})$.
We compose with the coassembly map \eqref{eqn:coassembly}  to obtain the {\it endomorphism torsion} map
 \begin{equation} \label{eqn:torsion}
 \tau^{\End}(t)\colon  K(\End_A^S) @>>> LK(A[t]_S)\, .
 \end{equation}
 
  \begin{notation} When the indeterminate $t$ or the decoration `$\End$' are understood, we drop them from that notation, writing
\[
\tau^{\End} = \tau^{\End}(t) = \tau(t) = \tau\, .
\]
\end{notation}

\begin{rem} On path components, $\tau$ defines a homomorphism
\[
\tau_\ast\: K_0(\End_A) \to K_1(A[t]_S) \times K_0(A[t]_S)\, .
\]
 The projection of $\tau_\ast$ 
 onto the second factor sends an object $(P,f)$ to the equivalence class of $P[t]_S$.
 The projection onto the first factor is given by the equivalence class in $K_1(A[t]_S)$ of the automorphism $t-f$. The latter can be viewed as the Reidemeister torsion of the short chain complex $P[t] @> t-f >> P[t]$ with respect to 
 extension $P[t] \to P[t]_S$. This chain complex is a resolution of the $A[t]$-module $P$ by free modules, so 
we can think of this as the Reidemeister torsion of the pair $(P,f)$. \end{rem}
 


\subsection{Examples} \label{subsec:torsion-examples} 
Let $A$ be a field and let $S\subset A[t]$ consist of all monic polynomials.
Then $A[t]_S = A(t)$ is the field of rational functions. 
There is an isomorphism
\begin{equation}\label{eqn:field_splitting}
\oplus_{p} K_\ast(A[t]/(p)) \to K_\ast(\End_A^S) = K_\ast(\End_A)\, ,
\end{equation}
where $p := p(t)$ ranges over the irreducible monic polynomials 
\cite{Quillen_exact}, \cite[p.~233]{Grayson_Quillen2},
\cite[th.~1.4]{Gersten}, \cite[th.~2.3]{Milnor_quadratic}.
The restriction to the $p$-th summand is the map $K_\ast(A[t]/(p)) \to K_\ast(\End_A)$ induced by the exact functor
which sends a finite dimensional vector space $V$ over $A[t]/(p)$ to the pair $(V,t\cdot)$.

Fix an irreducible monic polynomial $p$ and let $u\in A[t]/(p)$ be a nontrivial unit. 
Then $u \in K_1(A[t]/(p))$ is non-trivial. Its image  $\hat u\in K_1(\End_A)$ is the 
equivalence class of the automorphism $u$ of the endomorphism $t$ of $A[t]/(p)$. Since \eqref{eqn:field_splitting} is an isomorphism, $\hat u$ is non-trivial.



Let $\tau_\ast\: K_1(\End_A) \to K_2(A(t))$ be the homomorphism induced by the endomorphism torsion map $\tau$.

\begin{prop} \label{prop:non-trivial} The element $\tau_\ast(\hat u) \in K_2(A(t))$ is non-trivial.
\end{prop}

\begin{proof} As $\hat u$ is non-trivial, 
the result is an immediate consequence of Theorem \ref{thm:torsion-as-section} below
which says that the endomorphism torsion map $\tau\: K(\End_A) \to \Omega K(A(t))$ 
induces a section to the boundary map $\partial$
 in the localization sequence 
 \[
 0 \to K_2(A[t]) \to K_2(A(t)) @>\partial >> K_1(\End_A) @>>> 0\, .\qedhere
 \] 
\end{proof}

\subsection{Examples in higher dimensions} 
Let $A = \Bbb Q$, and $p(t) \in \Bbb Q[t]$ be an irreducible monic polynomial. Then $\Bbb Q[t]/(p)$ is a number field. Let $r_1$ denote its number of distinct embeddings into $\Bbb R$ and let $r_2$ its number of distinct conjugate pairs of embeddings into $\Bbb C$. Then by a theorem of Borel \cite{Borel}, for $n \ge 2$ there is an isomorphism of vector spaces
\begin{align*} 
K_n(\Bbb Q[t]/(p)) \otimes \Bbb Q  \cong 
\begin{cases}
0  & n \text{ even};\\ 
\Bbb Q^{r_1+r_2}  & n \equiv 1 \mod 4; \\
\Bbb Q^{r_2} &  n \equiv 3 \mod 4.\\
\end{cases}
\end{align*}
For example, let $p(t) = t^2 + bt+c$ have negative discriminant $D := b^2 - 4c$. Then 
$\Bbb Q[t]/(p) = \Bbb Q(\sqrt{D})$ is a quadratic number field with $r_1 = 0$ and $r_2 =1$. Hence,
the abelian group $K_{2n+1}(\Bbb Q[t]/(p))$ has rank one for $n \ge 1$.  Let $x\in K_{2n+1}(\Bbb Q[t]/(p))$ be an
element of infinite order and let 
\[
\hat x\in K_{2n+1}(\End_{\Bbb Q})
\] be its image with respect to the homomorphism $K_{2n+1}(\Bbb Q[t]/(p)) \to K_{2n+1}(\End_{\Bbb Q})$. Then $\hat x$ also has infinite order (cf.~\eqref{eqn:field_splitting}).
By Theorem \ref{thm:torsion-as-section} below, we infer that the endomorphism torsion
$\tau_\ast(\hat x) \in K_{2n+2}(\Bbb Q(t))$ has infinite order.

\section{The higher zeta function} \label{sec:zeta}
If $(P,f)$ is an object of $\End_A^S$ then we can define a different endomorphism of finitely generated projective $A[t]$-modules
\begin{equation}\label{dual_char_sequence}
P[t] @> 1-ft >> P[t]\, .
\end{equation}
Recall the multiplicative subset $T\subset A[t]$ of \S\ref{subsec:some-localizations}.
\begin{lem}\label{zeta_iso}
	After localizing at $T$, the map $1-ft$ is an isomorphism.
\end{lem}
\begin{proof}
	Before localization, the map is injective with cokernel $f^{-1}P$, with $t$ acting by $f^{-1}$. Therefore after localizing at $T$, it is still injective and its cokernel is $T^{-1}f^{-1}P$. So it suffices to show that some polynomial in $T$ annihilates $f^{-1}P$. By assumption some $p(t) \in S$ annihilates $P$, so we take $\tilde p(t) = t^np(t^{-1}) \in T$. The action of this polynomial on $f^{-1}P$ is by $\tilde p(f^{-1}) = f^{-n}p(f)$, which vanishes because $p(f)$ vanishes on $P$.
\end{proof}
Let
\[
\zeta_f(t) = (1-ft)^{-1}\colon P[t]_T @> \cong >> P[t]_T
\] 
denote the inverse map. Again the rule $(P,f) \mapsto (P[t]_T,\zeta_f(t))$ defines an exact functor 
\[
\End_{A}^S  @>>> \Aut_{A[t]_T} 
\]
and therefore a map of $K$-theory spaces $K(\End_A^S) \to K(\Aut_{A[t]_T})$.
We compose with the coassembly map \eqref{eqn:coassembly}  to obtain
\begin{equation} \label{eqn:zeta-function}
\zeta(t) \colon K(\End_A^S) \to L K(A[t]_T)\, .
\end{equation}

\begin{defn} The map \eqref{eqn:zeta-function} is called the $K$-theoretic {\it zeta function}.
\end{defn}

The operation $t\mapsto t^{-1}$ defines an involution $A(t) \to A(t)$ which restricts to a ring isomorphism 
$A[t]_T \to A[t^{-1}]_T$. As $t\in S$, we also have the homomorphism $A[t^{-1}]_T\to A[t]_S$. Assembling these,
we obtain a map on $K$-theory that we call $\zeta(t^{-1})$:
\[
\zeta(t^{-1})\: K(\End_A^S) @> \zeta(t) >> L K(A[t]_T) @> t \mapsto t^{-1} >> L K(A[t^{-1}]_T) \to L K(A[t]_S)\, .
\]

\begin{rem} The term ``zeta function'' is motivated by the observation that
	if we extend scalars along the embedding $A[t]_T \to A[[t]]$ we get the expression
	\[
	\zeta_f(t) = \sum_{k=0}^\infty f^kt^k  \: P[[t]] @> \cong >> P[[t]]\, ,
	\]
	where $f^kt^k$ is shorthand for $f^{\circ k}\otimes t^k$. Note that this further variant requires no torsion assumptions, defining a map
	\begin{equation} \label{eqn:torsionless-zeta-function}
	\zeta(t) \colon K(\End_A) \to L K(A[[t]])\, .
	\end{equation}
	The zeta function is therefore more general than endomorphism torsion.
\end{rem}

\begin{rem}\label{sec:deRham-Witt}
The zeta function is related to the topological de~Rham-Witt complex, studied by 
Hesselholt \cite{Hesselholt}, Betley and Schlichtkrull \cite{Betley-Schlichtkrull} and
Lindenstrauss and McCarthy \cite{Lindenstrauss-McCarthy}. In particular, the latter authors examine the map
\[ 
\zeta(t)\colon \widetilde K(\End_A) @>>> \Omega K(A[[t]])\, , \]
where $\widetilde K(\End_A)$ is the homotopy fiber of the forgetful map  $K(\End_A)\to K(A)$. Generalizing this to coefficients, this induces an equivalence of Taylor towers, whose homotopy limits are $W(A) = \TR(A)$. We therefore get a commutative diagram up to homotopy
\[ \xymatrix{
	\widetilde K(\End_A) \ar[rr]^-{\zeta} \ar[dr]_-{\textup{trc} \ } && \Omega K(A[[t]]) \ar[dl]^-{``\det"} \\
	& \TR(A).
} \]
where the lower-left map is the trace to $\TR$ as in \cite{Lindenstrauss-McCarthy,Malkiewich-et-al}.
However there doesn't appear to be an interesting lift of Milnor's identity to this setting, because we don't have an obvious ring map from $A[t]_S$ or $A[t^{-1}]_S$ to $A[[t]]$.
\end{rem}

\subsection{The map $t$}
To an object $(P,f)$ we can also associate the endomorphism
\[
t\: P[t] \to P[t]\, .
\]
Note that this does not depend on $f$, because it is the torsion of the zero endomorphism. As $t \in S$, it becomes an isomorphism once we extend scalars to $A[t]_S$. So the assignment $(P,f) \mapsto (P[t]_S,t)$ defines a map
\[
t\colon K(\End_A^S) \to K(A) \to K(\Aut_{A[t]_S}) @> c >> L K(A[t]_S)\, .
\]

\section{A generalization of Milnor's identity}
We are now in a position to generalize Milnor's identity on the level of $K$-theory. Recall that $\ast$ denotes the operation $LX \times_X LX \to X$ that composes loops with the same basepoint.

\begin{thm}\label{thm:milnor} Assume that $A$ is an associative unital ring and $S\subset A[t]$ is a multiplicative
subset of monic centric polynomials that contains $t$.
Then the composition
\[
K(\End_A^S) @>  \zeta(t^{-1})\ast \tau(t) >> L K(A[t]_S)
\]
is homotopic to the map $t$, as maps over $K(A) \to K(A[t]_S)$.
\end{thm}

\begin{proof} 
We first rewrite $\zeta(t^{-1})$. By the naturality of the coassembly map \eqref{eqn:coassembly}, it is equal to the composite
\[
K(\End_A^S) @> \zeta(t) >> K(\Aut_{A[t]_T}) @> t \mapsto t^{-1} >> K(\Aut_{A[t]_S}) @> c >> L K(A[t]_S)\, .
\]
The first map sends $(P,f)$ to $(P[t]_T,(1-ft)^{-1})$. The second map re-interprets this as $(P[t^{-1}]_T,(1-ft^{-1})^{-1})$ and then extends scalars to $A[t]_S$. On the underlying module this gives
\[ A[t]_S \otimes_{A[t^{-1}]_T} (A[t^{-1}]_T \otimes_A P) \cong A[t]_S \otimes_A P = P[t]_S. \]
Along this identification, the endomorphism $(1-ft^{-1})^{-1}$ is sent to the inverse of $1 \otimes (1-ft^{-1}) = t^{-1} \otimes (t-f)$. This inverse is $t \otimes (t-f)^{-1}$.

With this modification, the three maps $\tau(t),\zeta(t^{-1})$ and $t$ are all described as maps $K(\End_A^S) \to K(\Aut_{A[t]_S})$, followed by coassembly. After applying the forgetful map $K(\Aut_{A[t]_S}) \to K(A[t]_S)$, all three send $(P,f)$ to $P[t]_S$, up to canonical isomorphism. On $K$-theory spaces they are therefore canonically homotopic. In fact, by fixing models for the extension of scalars of each module, we can modify them up to homotopy so that their projections to $K(A[t]_S)$ strictly coincide. We therefore have three maps of the form $K(\End_A^S) \to K(\Aut_{A[t]_S})$ over the same map $K(\End_A^S) \to K(A[t]_S)$.

As a result of Lemma \ref{coassembly_respects_composition}, it is now enough to prove that $\zeta(t^{-1}) \circ \tau(t) = t$ as fiberwise maps
\[
K(\End_A^S) @>>> K(\Aut_{A[t]_S})
\]
over $K(A[t]_S)$. For an object $(P,f) \in \End_A^S$, we have shown earlier in the proof that $\zeta_f(t^{-1})$ produces the endomorphism 
 \[
 (t-f)^{-1} \otimes t\: P\otimes_A A[t]_S \to P\otimes_A A[t]_S \, .
 \]
On the other hand, $\tau_f(t)$ produces the endomorphism $(t-f) \otimes 1$. Therefore their composition is $1 \otimes t$ as desired.
%
\end{proof}

We now explain how Theorem \ref{thm:milnor} is a lift of Milnor's original identity. Note that 
Milnor only considers the case when $A$ is a field and $S$ consists of all monic polynomials. In this case the identity lies
in the units of the  ring of rational functions $A(t)$.

We argue as follows:
projecting from free loops to based loops, our result implies the same identity $\zeta(t^{-1})\ast \tau(t) \sim t$ as maps
\[ K(\End_A^S) @>>> \Omega K(A[t]_S)\, . \]
Taking $\pi_0$ gives the identity $\zeta(t^{-1}) \ast \tau(t) = t$ as maps of $K$-groups
\[ K_0(\End_A^S) @>>> K_1(A[t]_S)\, , \]
where $\ast$ is the abelian group structure on $K_1$. Applying the determinant map $K_1(R) \to R^{\times}$ makes our definitions of $\zeta(t^{-1})$ and $\tau(t)$ agree with their original definitions, namely the determinant of $(1-ft^{-1})^{-1}$ and $t-f$, respectively. We recover the identity
\[ \zeta(t^{-1})\tau(t) = t^{\chi(P)} \]
as maps
\[ K_0(\End_A^S) @>>> (A[t]_S)^{\times}\, . \]
\begin{rem}
	The notation $t^{\chi(P)} = \det(t)$ is slightly misleading because in general, $P$ is not necessarily a free module, so the determinant of $\cdot t$ on $P[t]$ might not be a monomial of the form $t^{\chi(P)}$. However it is guaranteed to be a monomial as soon as $A$ is an integral domain and therefore has no nontrivial idempotents; see \cite{Goldman}.
\end{rem}

For a different simplification, recall that $t\: K(\End_A) \to LK(A[t]_S)$ factors through the forgetful functor $K(\End_A) \to K(A)$. Consequently, if we take the identity from Theorem \ref{thm:milnor} and pass to homotopy fibers over $K(A) \to K(A[t]_S)$, the map $t$ becomes trivial (constant map to the basepoint) 
and so we get the simpler identity $\tilde \zeta(t^{-1})\ast \tilde \tau(t) \sim 0$, where $\tilde \zeta$ and $\tilde \tau$ refer to the restrictions of $\zeta$ and $\tau$ to maps
\[ \widetilde K(\End_A^S) @>>> \Omega K(A[t]_S)\, . \]
\begin{cor}\label{tau_zeta_inverses}
	$\tilde \zeta(t^{-1})$ and $\tilde \tau(t)$ are additive inverses in the homotopy category.
\end{cor}


\section{Chain complexes} \label{sec:chain_complexes}

In this section we describe a chain complex model for $K(\End^S_A)$. We introduce this model in order to construct endomorphism torsion for topological families of endomorphisms in \S\ref{sec:families}.
The construction and proof of Milnor's identity for the chain complex model is essentially identical to the constructions given above, but passing to chain complexes gives us the ability to relate the higher torsion to the boundary map of the $K$-theory localization sequence; see Appendix \ref{sec:boundary}.


For simplicity, we consider chain complexes in non-negative degrees only, though the results of this section also hold for chain complexes in all degrees.

Again let $A$ be an associative unital ring. Let $\Ch(A)$ denote the category of chain complexes of $A$-modules and chain maps. Let $\Ch^{A\text{-}\proj}(A)$ denote the subcategory of degreewise projective complexes. This has the structure of a Waldhausen category: a cofibration is a levelwise injective map with projective cokernel, and a weak equivalence is a quasi-isomorphism. 

Given a map $f\colon P. \to Q.$ of $A$-chain complexes, we let $T(f).$ denote its mapping cylinder. It is characterized by the property that maps $T(f). \to D.$ correspond to pairs of maps $g\colon P. \to D.$, $h\colon Q. \to D.$ and a chain homotopy $g \sim hf$. Similarly let $C(f). = T(f)./P.$ be the mapping cone.

If $f\colon P. \to P.$ is an endomorphism, let $T^+(f).$ denote its mapping telescope. Let $T^{-}(f).$ denote the reverse mapping telescope, constructed as $T^+(f)$ but appending each new segment to the front of the cylinder, rather than the back. A map $T^+(f) \to D.$ consists of maps $g^i\colon P. \to D.$ for $i \geq 0$ and chain homotopies $g^i \sim fg^{i-1}$. A map $T^-(f) \to D.$ consists of the same except the maps are indexed over $i \leq 0$. Note that the canonical map $P. \to T^+(f).$ is a localization by $f$ and the canonical maps $P. \to T^-(f). \to P.$ are quasi-isomorphisms.

\begin{rem}
	The reverse mapping telescope $T^{-}(f)$ has an endomorphism $\sh^{-1}$ that shifts one slot to the left. With respect to this $A[t]$-action the collapse $T^-(f). \to P.$ is $A[t]$-linear, in fact it is essentially a rewriting of the characteristic sequence \eqref{eq:char-sequence} for chain complexes. Similarly, $T^+(f)$ has an endomorphism that shifts one slot to the right, making it a chain complex version of the dual characteristic sequence \eqref{dual_char_sequence} that we used to define the zeta function.
\end{rem}

We say a complex of $A$-modules is {\it strictly perfect} if it is bounded (i.e. nonzero in only finitely many degrees) and finitely generated projective over $A$ at each level. A complex is {\it perfect} if it is quasi-isomorphic to a strictly perfect complex. This condition is equivalent to being compact in the derived category, so it is preserved by extensions, retracts, kernels, and cokernels. The perfect complexes define a Waldhausen subcategory
\[ \Ch_{\perf}(A) := \Ch^{A\text{-}\proj}_{A\text{-}\perf}(A) \subseteq \Ch^{A\text{-}\proj}(A)\, . \]

Next we consider chain complexes with an endomorphism. This is the same thing as an $A[t]$-chain complex, but a priori, there is more than one way to characterize when such a complex is perfect over $A$, or is $S$-torsion. The following lemmas eliminate this ambiguity.
%

\begin{lem}\label{perfect_is_perfect}
	The following two conditions are equivalent. We say an $A[t]$-complex $P.$ is {\it $A$-perfect} if either one holds.
\end{lem}

\begin{itemize}
	\item There is a zig-zag of $A[t]$-linear quasi-isomorphisms to an $A$-strictly perfect complex.
	\item There is a zig-zag of $A$-linear quasi-isomorphisms to an $A$-strictly perfect complex.
\end{itemize}

\begin{proof}
	Clearly the first implies the second. Suppose $P.$ is an $A[t]$-chain complex satisfying the second condition, and denote the action of $t \in A[t]$ by $f\colon P. \to P.$. Without loss of generality the zig-zag is a single quasi-isomorphism $g\colon C. \to P.$ where $C.$ is a strictly perfect $A$-chain complex. Lift $f$ along this equivalence to an $A$-linear endomorphism $f'\colon C. \to C.$ such that the following square commutes up to a chosen chain homotopy.
	\[ \xymatrix{
		C. \ar[r]^-{f'} \ar[d]_-g & C. \ar[d]^-g \\
		P. \ar[r]_-f & P.
	} \]
	Let $T^{-}(f')$ be the reverse mapping telescope discussed above, as an $A[t]$-chain complex. Let $T^{-}(f') \to C.$ be the collapse onto the end and define $T^{-}(f') \to P.$ by sending the $i$th copy of $C.$ to $P.$ by $f^{|i|} \circ g$, and extending over the 1-skeleton using the chosen chain homotopy for the square above (composed with iterates of $f$). Both of these maps are $A[t]$-linear quasi-isomorphisms, giving the desired zig-zag from $P.$ to the strictly perfect $A$-chain complex $C.$.
\end{proof}

\begin{lem}\label{small_to_big_perf}
	If $P.$ is $A$-perfect then it is $A[t]$-perfect.
\end{lem}

\begin{proof}
	In light of the previous lemma, without loss of generality $P.$ is $A$-strictly perfect and we must show it admits a quasi-isomorphism from an $A[t]$-strictly perfect complex. Simply take the reverse mapping telescope $T^-(f). \to P.$; it is obtained from the bicomplex
	\[
	P.[t]@> t-f >> P.[t]
	\]
	(cf.~\eqref{eq:char-sequence}) and is therefore $A[t]$-strictly perfect.
\end{proof}

Again, fix a multiplicative central set of monic polynomials $S \subseteq A[t]$.

\begin{lem}\label{lem:torsion_is_torsion}
	For an $A[t]$-perfect complex $P.$, the following five conditions are equivalent. We say $P.$ is {\it $S$-torsion} if any of them hold.
\end{lem}

\begin{itemize}
	\item There exists $p(t) \in S$ such that $p(f)$ is equivalent to 0 in the $A[t]$-derived category.
	\item There exists $p(t) \in S$ such that $p(f)$ is equivalent to 0 in the $A$-derived category.
	\item There exists $p(t) \in S$ such that $p(f) = 0$ on $H_*(P.)$.
	\item There exists $p(t) \in S$ such that $p(f)^{-1}H_*(P.) = 0$.
	\item There exists $p(t) \in S$ such that $p(f)^{-1}P.$ is acyclic.
\end{itemize}

\begin{proof}
	Clearly each one implies the next. So assume that $p(f)^{-1}P.$ is acyclic. Without loss of generality $P.$ is $A[t]$-strictly perfect. The localization $p(f)^{-1}P.$ can be modeled as the mapping telescope $T^+(p(f))$, so this telescope is acyclic. Therefore the inclusion $P. \to T^+(p).$ admits a chain homotopy to zero, giving an extension to $C(0). \to T^+(p).$, where $C(0).$ is the mapping cone of the zero map $0. \to P.$. The mapping cone $C(0).$ is also $A[t]$-strictly perfect, so this extension factors through a finite stage of the mapping telescope $T^n.$. This finite stage admits a quasi-isomorphism to $P.$ by collapsing onto the end, and the composite $P. \to T^n. \to P.$ is $p(f)^n$. However this factors through the acyclic complex $C(0).$, so it is zero in the $A[t]$-derived category. Therefore $p(t)^n \in S$ acts as 0 on $P.$ in the $A[t]$-derived category.
\end{proof}

\begin{lem}\label{lem:big_to_small_perf}
	If $P.$ is $A[t]$-perfect and $S$-torsion then $P.$ is $A$-perfect.
\end{lem}

\begin{proof}
	Let $p(t) \in S$ be a monic polynomial such that $p(f)$ is zero in the $A[t]$-derived category. Let $R.$ be the two-term complex $A[t] @> p >> A[t]$, whose $A[t]$-module structure extends in a unique way to the structure of a DGA. Note that $R.$ is equivalent as an $A$-complex to the one-term complex $A[t]/p \cong A^{\oplus \deg p}$, and therefore $R.$ is perfect over $A$. Since $P$ is perfect over $A[t]$, the extension of scalars $P. \otimes_{A[t]} R.$ is therefore perfect over $A$ as well. This latter complex is isomorphic to $C(p).$, the mapping cone of $p(f)\colon P. \to P.$. Therefore $C(p).$ is perfect over $A$.
	
	Since $P.$ is $A[t]$-projective and $p(f)$ is zero in the $A[t]$-derived category, $p(f)\colon P. \to P.$ is chain null homotopic. Picking a null homotopy defines an $A[t]$-linear map $C(p). \to P.$ that splits the inclusion $P. \to C(p).$. Therefore $P.$ is an $A$-linear retract of the $A$-perfect complex $C(p).$, so $P.$ is $A$-perfect as well.
\end{proof}

\begin{cor}\label{perfect_is_torsion}
	If $P.$ is $S$-torsion then it is $A[t]$-perfect iff it is $A$-perfect.
\end{cor}

Now that the ambiguity has been removed from the definitions, we explain how these chain complex models are equivalent to the ones of the previous sections. Recall that the functor that sends each module to the corresponding chain complex supported in degree 0 gives an equivalence
\[ 
K(A) = \Omega |iS.P(A)| @> \sim >> \Omega|wS.\Ch^{A\text{-}\proj}_{A\text{-}\perf}(A)|\, .
\]
This is by an application of Waldhausen's ``sphere theorem'' \cite[thm.~1.71]{Wald_1126}. The main hypothesis can be verified by a method described in Appendix \ref{sec:sphere_theorem} (Proposition \ref{sphere_hyp}). By a standard abuse of notation we let $K(A)$ refer to either of these spaces.

Once we consider endomorphisms and torsion, we have more choices of model, illustrated in the diagram below. 
\[ \xymatrix{
	\Ch^{A[t]\text{-}\proj}_{A[t]\text{-}\perf}(A[t]) \ar@{<-}[r] &
	\Ch^{A[t]\text{-}\proj}_{A\text{-}\perf}(A[t]) \ar[r]^-\sim & \ar@/^1.5em/[l]^-\sim
	\Ch^{A\text{-}\proj}_{A\text{-}\perf}(A[t]) & \ar[l]_-\sim
	\End_A \\
	\Ch^{A[t]\text{-}\proj}_{A[t]\text{-}\perf}(A[t])^S \ar@{=}[r]^-\sim &
	\Ch^{A[t]\text{-}\proj}_{A\text{-}\perf}(A[t])^S \ar[r]^-\sim & \ar@/^1.5em/[l]^-\sim
	\Ch^{A\text{-}\proj}_{A\text{-}\perf}(A[t])^S & \ar[l]_-\sim
	\End_A^S
} \]
The straight maps are all inclusions. The curved map is the reverse mapping telescope functor from the proof of Lemma \ref{small_to_big_perf}. The $S$-superscripts denote the subcategory of $S$-torsion complexes in the sense of Lemma \ref{lem:torsion_is_torsion}.

\begin{prop}\label{four_models_for_endomorphisms}
	Every map marked $\sim$ induces an equivalence on $K$-theory.
\end{prop}

\begin{proof}
	The leftmost equivalence is an equality of categories by Corollary \ref{perfect_is_torsion}. The arrows in the center are inverses up to equivalence by the proof of Lemma \ref{small_to_big_perf}. Alternatively, the straight arrows are equivalences using the approximation theorem. The remaining inclusions on the right-hand side are by Waldhausen's sphere theorem \cite[thm.~1.71]{Wald_1126}. In each case the main hypothesis of the theorem is not obvious because the modules are not finitely generated projective over $A[t]$, but the hypothesis does hold (Propositions \ref{sphere_hyp_end} and \ref{sphere_hyp_tors}).
\end{proof}

%

\begin{cor}\label{cor:perfect_is_perfect} 
	The exact functor sending $(P,f)$ to the $A[t]$-complex $P[t]@> t-f >> P[t]$ induces an equivalence
	\begin{align*}
		K(\End_A^S) @> \sim >> \, & K(\Ch_{\perf}(A[t])^S).
	\end{align*}
\end{cor}

The construction of higher torsion $\tau$ and the zeta function $\zeta$ can be defined as in the previous sections, only taking $P$ to be an $A$-perfect complex with an endomorphism, instead of a finitely generated projective $A$-module with an endomorphism.
\begin{thm}\label{thm:milnor_chain_complexes} The definitions of $\tau$ and $\zeta$ extend in a natural way along the inclusions of categories
	\[
	\End_A^S \subseteq \Ch^{A\text{-}\proj}_{A\text{-}\perf}(A[t])^S, \qquad
	P(A[t]_S) \subseteq \Ch_{\perf}(A[t]_S)
	\]
and the composition
	\[
	K(\Ch^{A\text{-}\proj}_{A\text{-}\perf}(A[t])^S) @>  \zeta(t^{-1})\ast \tau(t) >> L K(\Ch_{\perf}(A[t])_S)
	\]
	is homotopic to the map $t$, as maps over $K(\Ch_{\perf}(A)) \to K(\Ch_{\perf}(A[t])_S)$.
\end{thm}

\begin{proof}
	The most significant change in defining the maps is that the map $t-f$, respectively $1-ft$, is a quasi-isomorphism rather than an isomorphism, because its cofiber is acyclic, not identically zero. When proving this for $1-ft$, it is illuminating to refer to $f^{-1}P$ as the extension of scalars $P \otimes_{A[t]} A[t,t^{-1}]$, and then to observe that inverting $p(f)$ and $\tilde p(f^{-1})$ are the same localization. Then we define the automorphism category to consist of chain complexes and self-quasi-isomorphisms.
	
	As a result, $\zeta$ is not an isomorphism, so it can't be inverted. To correct for this, instead of defining $\zeta$ as $(1-ft)^{-1}$, we define it as $(1-ft)$ and at the very end apply the flip map to the free loop space to accomplish the inversion.
	
	It is possible to modify the proof of the identity accommodate this weaker setup, but it is faster to deduce the identity directly from Theorem \ref{thm:milnor}, since all the $K$-theory spaces involved are equivalent to those appearing in Theorem \ref{thm:milnor}.
\end{proof}

\section{The non-linear setting} \label{sec:non-linear}


For a space $X$, we define a Waldhausen category
$\End_X$. We construct a linearization map on $K$-theory $K(\End_X) \to K(\End_{\Bbb Z[\pi]})$, where
 $\pi = \pi_1(X)$. 

For a fixed ring homomorphism $\Bbb Z[\pi] \to A$ and any multiplicative system of monic centric polynomials
 $S\subset A[t]$, we associate a full subcategory $\End^S_X \subset \End_X$ of $S$-torsion objects. 
We then exhibit a linearization map $K(\End^S_X) \to K(\End^S_A)$.

\subsection{Endomorphisms of retractive spaces}
For each topological space $X$, let $T(X)$ be the category of retractive spaces $Y$ over $X$. There is a Quillen model category structure on $T(X)$ in which the weak equivalences are weak homotopy equivalences and the cofibrations are the retracts of relative cell complexes. This makes the subcategory $R(X) \subset T(X)$ of cofibrant objects into a Waldhausen category.

An object $Y\in T(X)$ is {\it finite} if it is built up from the zero object by a finite number of cell attachments. It is {\it homotopy finite} if it is weakly equivalent to a finite object, and {\it finitely dominated} if it is a retract of a homotopy finite object.

Let $R_{\fd}(X)\subset R(X)$ be the full subcategory on the finitely dominated objects. Then $R_{\fd}(X)$ together with the above cofibrations $\text{co}R_{\fd}(X)$ and weak equivalences $vR_{\fd}(X)$ forms a Waldhausen category.

Let $\End_X = \End(R_\fd(X))$ denote the category of endomorphisms of objects in $R_\fd(X)$. As before, this inherits a Waldhausen structure in which a weak equivalence or cofibration is determined by forgetting the endomorphisms.

\subsection{The $T.$ construction}
We will define a linearization map back to $\End_{\Bbb Z[\pi_1(X)]}$, but for this we need a variant of the $S.$ construction due to Thomason \cite[p.~334]{Wald_1126}. To each a Waldhausen category $(C,\text{co}C,wC)$ we may associate a simplicial category $wT.C$. The objects of $wT_kC$ are sequences of cofibrations
\[
A_\bullet := A_0 \rightarrowtail A_1 \rightarrowtail \cdots \rightarrowtail  A_k.
\]
A morphism $A_\bullet\to B_\bullet$ consists of maps $A_i \to B_i$ for $0\le i \le k$ such that each of the squares
\[
\xymatrix{
	A_i \ar@{>->}[r] \ar[d] & A_{i+1}\ar[d] \\
	B_i \ar@{>->}[r] & B_{i+1}
}
\]
is a homotopy cocartesian square, i.e. the square commutes and the map $B_i \cup_{A_i} A_{i+1} \to B_{i+1}$ is a weak equivalence. The $i$-th face operator drops $A_i$ from the sequence and the $i$-th degeneracy operator inserts the identity map. There is a chain of homotopy equivalences
\[
wT.C @< \simeq << wT.^{\!\! +}C @> \simeq >> wS.C
\]
where the middle simplicial category is defined in a way similar to $wT.C$ but where quotient data is included.
The left equivalence is given by forgetting quotient data and the right one is given by mapping the sequence
$A_\bullet$ to the sequence 
$A_1/A_0 \rightarrowtail A_2/A_0  \rightarrowtail \cdots$.

\subsection{Linearization} \label{sec:linearize}
Now assume that $X$ is based, path-connected and has a universal cover $\widetilde X \to X$. Let $\pi = \pi_1(X)$ be the fundamental group of $X$. Under these assumptions we define a linearization map
\[
\cal L\: \End_X \to \End_{\Bbb Z[\pi]}\, .
\]
For each $Y \in R_\fd(X)$, let $\widetilde Y$ denote the pullback
\[ \xymatrix @R=1.5em{
	\widetilde Y \ar[d] \ar[r] & Y \ar[d] \\
	\widetilde X \ar[r] & X.
} \]
Then the reduced singular chain complex $C.(\widetilde Y,\widetilde X)$ 
is a perfect $\Bbb Z[\pi]$-chain complex. The functor $\cal L$ is then the operation that assigns $(Y,f)$ to $(C.(\widetilde Y,\widetilde X),f_*)$.

The functor $\cal L$ preserves cofibrations and weak equivalences but does not preserve pushouts, so it does not induce a map on the $S.$ construction. However $\cal L$ preserves homotopy cocartesian squares, hence it induces a map of simplicial categories
\[
wT.\End_X \to wT.\End_{\Bbb Z[\pi]}
\]
Taking realization and loops, we obtain the linearization map
\[
{\cal L}\: K(\End_X) \to  K(\End_{\Bbb Z[\pi]})\, .
\]

\subsection{Torsion endomorphisms}\label{sec:torsion_of_space}
Let $\Bbb Z[\pi] \to A$ be any ring homomorphism and let $S\subset A[t]$ be a multiplicative subset of monic centric polynomials. To define $S$-torsion for endomorphisms of retractive spaces we simply apply linearization. 

\begin{defn} An object $(Y,f)\in \End_X$ is {\it $S$-torsion} if $\cal L(Y,f) \in \End_{\Bbb Z[\pi]}$ becomes, after extending scalars to $A$, an $S$-torsion endomorphism of chain complexes in the sense of Lemma \ref{lem:torsion_is_torsion}. Let
\[
	\End_X^S \subset \End_X
\] 
be the full subcategory consisting of the $S$-torsion objects. 
\end{defn}

It is readily verified that $\End_X^S$ inherits the structure of a Waldhausen category.
Then linearization followed by extension of scalars induces a functor
\[
\cal L\: \End_X^S \to \End_A^S\, ,
\]
and as above this induces a map on $K$-theory,
\begin{equation} \label{eqn:linearization_map}
\cal L\:  K(\End_X^S) \to  K(\End_A^S)\, .
\end{equation}

\begin{rem} Our definition of $\End^S_X$ differs from that of Levikov \cite{Levikov}, who
	defines $S$-torsion at the level of suspension spectra rather than at the chain level.
\end{rem}

We define the endomorphism torsion $\tau$ and zeta function $\zeta$ for any object $(Y,f) \in \End_X^S$ by applying linearization \eqref{eqn:linearization_map}. These give classes in $\Omega K(A[t]_S)$ satisfying the same identity as before. When $A$ is a field, we recover the original form of Milnor's identity \cite[p.~123]{Milnor_inf_cyclic}.

\section{Families of endomorphisms} \label{sec:families}

In this section we give a few examples of higher classes in endomorphism $K$-theory $K(\End_X^S)$, which by the previous section produce higher classes in $K(\End_A^S)$ and therefore have higher endomorphism torsion.

\subsection{Higher torsion for families} \label{sec:correspondence}
The essential idea is that to give a map $B \to K(\End_X)$, where $B$ is some topological space, it is enough to give a map
\[ B \to |w.\End_X|\, , \]
and such maps correspond to fibrations of the following form.
\begin{defn}
	A $B$-family of endomorphisms over $X$ is a retractive space $E$ over $B \times X$, and an endomorphism $f\colon E \to E$ of retractive spaces, such that the projection $p\colon E \to B$ is a fibration and each fiber is a finitely dominated retractive space over $X$. We sometimes abbreviate this by $(p,f)$ or just $E$, but the other data is understood.
\end{defn}

\begin{prop}\label{families_correspondence}
Such families, up to the evident notion of weak equivalence, correspond to homotopy classes of maps $B \to |w.\End_X|$.
\end{prop}

We omit the proof, which is by classical arguments as in \cite{may_classifying}. However we describe the correspondence itself. Given a map to $|w.\End_X|$, we pull back the universal family over $|w.\End_X|$ to produce a family over $B$. This universal family is defined by taking the bar construction $B(\iota,w\End_X,*)$ where $\iota\colon w\End_X \to \End_X$ is the functor taking every retractive space $Y$ over $X$ to itself. By the proof of Quillen's Theorem B, the projection of this bar construction back to $B(\iota,w\End_X,*) = |w.\End_X|$ is a quasifibration, every fiber of which is a retractive space over $X$ with an endomorphism. In other words, once we replace the map to $X \times |w.\End_X|$ by a fibration, it is a $|w.\End_X|$-family. Then this is the universal family in the sense that any other family is the pullback of this one along a map $B \to |w.\End_X|$.

\begin{rem}
We can also give an explicit inverse to this operation. Given a $B$-family of endomorphisms, we can regard it as a point in the classifying space of all $B$-families up to weak equivalence. This classifying space, as a functor of $B$, is a contravariant homotopy functor and therefore has a coassembly map, see e.g. \cite[\S 5]{umkehr} and \cite[\S 5]{coassembly}. This coassembly map sends the given family back to a map $B \to |w.\End_X|$. Using the universal property of coassembly, we identify this as the inverse to the above correspondence.
\end{rem}

%

For each $B$-family of endomorphisms $E \to B$ we call the resulting compositions
\begin{equation*}\label{eqn:torsion-inv}
\xymatrix @R=-0.3em{
	&&& \Omega K(A[t]_S) \\
	B \ar[r] & K(\End^S_X) \ar[r]^-{{\cal L}_A} & K(\End^S_{A}) \ar[ur]^-\tau \ar[dr]_-\zeta & \\
	&&& \Omega K(A[t]_T)
}
\end{equation*}
the {\it higher endomorphism torsion} $\tau(E)$ and {\it higher zeta function} $\zeta(E)$ of the family.

\subsection{Examples} \label{subsec:examples}

We give a few examples of higher endomorphism torsion where the base space $B = S^1$ is the circle, $X = *$ is a point, and $A = \Bbb Q$. In this case the interesting part of the endomorphism torsion lies in the group $K_2(\Bbb Q(t))$. To give a map $S^1 \to |w.\End_{*}^S|$ it suffices to select a based space $Y$ with two commuting self-maps $f$ and $\theta$, such that $\theta$ is a weak equivalence.

%

\begin{ex}\label{ex:prototype} Let $n \geq 1$ and consider any based self-equivalence
\[ \theta\: S^n\vee S^n \to S^n \vee S^n \]
whose action on $n$th homology is the invertible matrix
\begin{equation}\label{eqn:monodromy}
Q = \begin{bmatrix}
0& - 1\\
1 & 1
\end{bmatrix}
\end{equation}
Set $f=\theta$, so that $f$ and $\theta$ trivially commute. This determines the $S^1$-family of endomorphisms 
\[
p\:E\to S^1, \quad f\: E \to E.
\]
To be more precise, $E$ is the mapping torus of $\theta$, made into a fibration with fiber $S^n \vee S^n$. In fact it is a non-trivial homology circle with $\pi_1(E) \cong \Bbb Z$ (when $n \geq 2$). Let $u$ denote the resulting homotopy class
\[
u\: S^1 @> >> K(\End_\ast)\, .
\]
An unraveling of the construction shows that the composition
\begin{equation*} \label{eqn:torsion-class}
S^1 @> u >> K(\End_\ast) @> L_{\Bbb Q} >> K(\End_{\Bbb Q})\, ,
\end{equation*}
where $L_{\Bbb Q}$ is the linearization map of \S\ref{sec:linearize}, is homotopic to the class represented by the loop
\begin{equation}\label{eqn:loop-example}
\theta_\ast\: (\widetilde C.(S^n\vee S^n),f_\ast)   @>\simeq >> (\widetilde C.(S^n\vee S^n),f_\ast) \, ,
\end{equation}
where $\widetilde C.(S^n\vee S^n)$ is the reduced total singular complex over $\Bbb Q$ of  $S^n\vee S^n$. We then simplify this by the zig-zag of quasi-isomorphisms
\[ \xymatrix @R=0.6em {
	\vdots \ar[d] & \vdots \ar[d] & \vdots \ar[d] \\
	\widetilde C_{n+1}(S^n \vee S^n) \ar[d] & \ar[l]_-= \widetilde C_{n+1}(S^n \vee S^n) \ar[d] \ar[r] & 0 \ar[d] \\
	\widetilde C_n(S^n \vee S^n) \ar[d] & \ar[l] \widetilde Z_n(S^n \vee S^n) \ar[d] \ar[r] & \Z \oplus \Z \ar[d] \\
	\widetilde C_{n-1}(S^n \vee S^n) \ar[d] & \ar[l] 0 \ar[d] \ar[r] & 0 \ar[d] \\
	\vdots & \vdots & \vdots
} \]
to the chain complex $(\Z \oplus \Z)[n]$ with $\theta_\ast,f_\ast$ acting by the matrix \eqref{eqn:monodromy}. Alternatively, we can write this chain complex as $(\Bbb Q[t]/(t^2-t+1))[n]$ with $\theta_\ast,f_\ast$ acting by multiplication by $t$.
%
%
The resulting class $\theta_\ast \in K_1(\End_{\Bbb Q})$ is therefore the non-trivial unit
\[ (-1)^n t\in (\Bbb Q[t]/(t^2-t+1))^\times \cong K_1(\Bbb Q[t]/(t^2-t+1)) \]
by \S\ref{subsec:torsion-examples}. 
By Proposition \ref{prop:non-trivial}, we infer that 
 the loop \eqref{eqn:loop-example} is non-trivial, and conclude the following.
\end{ex}

\begin{thm} \label{thm:non-trivial-example} The  endomorphism torsion invariant of the above pair $(p,f)$ is  
non-trivial in  $K_2(\Bbb Q(t))$.
\end{thm}


\begin{ex}\label{ex:reflection} Another example is given by the matrix
\[
 R\, :=\,  \begin{bmatrix}
0& - 1\\
1 & 0
\end{bmatrix}\, .
\]
In this case the associated fibration $p\: E\to S^1$ with fiber $S^n \vee S^n$ is a rational homology circle:
\[
H_\ast(E) \cong 
\begin{cases}
\Bbb Z & \ast=0,1, \\
\Bbb Z/2 & \ast = n,\\
0 & \text{otherwise.}
\end{cases}
\]
As in the previous example, we use the matrix $R$ for both the clutching data $\theta$ and the fiberwise endomorphism $f$. The unit is then $t\in \Bbb Q[t]/(t^2+1)$, giving rise to
a non-trivial element of $K_2(\Bbb Q(t))$ which is the endomorphism torsion of $(p,f)$. 

In this example, the fibration can be chosen as a fiber bundle with monodromy
 $\theta\: S^n \vee S^n \to S^n \vee S^n$ defined
by $\theta(x,y) = (r(y),x)$, in which $r\: S^n \to S^n$ is the reflection $r(y_0,\dots,y_n) = (-y_0,\dots,y_n)$.
Observe that $\theta$ is 4-periodic.
\end{ex}

\begin{ex}\label{ex:trefoil}  Consider the trefoil knot $K\:S^1 \subset S^3$. This is a fibered knot, so one has a smooth fiber bundle
\[
p\: E\to S^1 
\]
where $E$ is the knot complement and the fiber $V$,  a Seifert surface for $K$,
 is a torus with an open disk removed.
The geometric monodromy $\theta\: V\to V$ is a diffeomorphism which restricts to the identity on $\partial V = S^1$. Moreover,
$\theta$ is 6-periodic \cite{Gordon_fibered}. 
With respect to a suitable choice of basis for $H_1(V;\Bbb Z)$, the
homological monodromy matrix is given by 
\eqref{eqn:monodromy}.
Since $\theta$ commutes with itself, it induces a fiberwise diffeomorphism
 $f\: E\to E$ that covers the identity map of $S^1$.  Again, the endomorphism torsion $\tau(p,f) \in K_2(\Bbb Q(t))$ is non-trivial, by the argument appearing in Example \ref{ex:prototype}.\end{ex}

\begin{ex} \label{ex:closed} There is a variant of the Example \ref{ex:reflection} in which the fibration $E\to S^1$ 
is a smooth fiber bundle over the circle
with fiber $S^n \times S^n$ and monodromy map  $\theta\: S^n \times S^n \to S^n \times S^n$,
defined by $\theta(x,y) = (r(y),x)$. The endomorphism torsion invariant is again non-trivial; we omit the details.
\end{ex}

\subsection{Higher dimensional examples\label{subsec:higher}} 
We briefly outline how one can indirectly construct higher dimensional examples with non-trivial endomorphism torsion in 
$K_\ast(\Bbb Q(t))$ in every even degree $\ge 4$. The details of this construction depend on a recent result due to Andrew Salch and go well beyond the scope of the current work.

Let $\Bbb S$ be the sphere spectrum. Let $\Bbb S[i]$ be a {\it Gaussian sphere}, i.e. a structured  ($A_\infty$), connective
$\Bbb S$-algebra such that
\begin{itemize}
\item The ring $\pi_0(\Bbb S[i])$ is isomorphic to the Gaussian integers $\Bbb Z[i]$;
\item the homology $H_\ast(\Bbb S[i];\Bbb Z)$ is concentrated in degree zero;
\item $\Bbb S[i]$ is weakly equivalent to $\Bbb S \vee \Bbb S$ as an $\Bbb S$-module.
\end{itemize}
Such an $\Bbb S$-algebra exists \cite{Salch} but we assume this without proof.

For a structured ring spectrum $R$ we let $M_k(R)$ be the derived $R$-module 
endomorphisms of a wedge of $k$-copies of $R$. Then
$\gl_k(R)$ is defined by as a pullback 
\[
\xymatrix{
\gl_k(R) \ar[r] \ar[d] & M_k(R) \ar[d]\\
\gl_k(\pi_0(R)) \ar[r] & M_k(\pi_0(R))
}
\]
i.e., those endomorphisms which are invertible up to homotopy.  
Then $\gl_k(R)$ is a topological monoid and one has stabilization maps $\gl_k(R) \to \gl_{k+1}(R)$. Let $B\gl(R)$
be the colimit of $B\gl_k(R)$ under stabilization.
The algebraic $K$-theory of $R$ is 
\[
K(R) =  K_0(\pi_0(R)) \times B\gl(R)^+ 
\]
(cf.~\cite[p.~67]{Blumberg-et-al}).
When $R = \Bbb S[i]$ the map 
$$
K(\Bbb S[i]) \to K(\Bbb Z[i])
$$
is a rational homotopy equivalence \cite[lem.~2.4]{Land-Tamme}.

According to \cite{Borel}, the abelian group
$K_{2n+1}(\Bbb Z[i]) \cong K_{2n+1}(\Bbb Q[i])$ has rank one for $n \ge 1$.  
Consequently, the abelian group $K_{2n+1}(\Bbb S[i])$ has rank one for $n\ge 1$.
Represent a non-torsion element
$x\in K_{2n+1}(\Bbb S[i]) $  by a map $S^n \to B\gl_k(\Bbb S[i])^+$
for $k$ sufficiently large, and form the homotopy pullback
\[
\xymatrix{
\Sigma \ar[r]\ar[d] & B\gl_k(\Bbb S[i])\ar[d] \\
S^{2n+1} \ar[r] & B\gl_k(\Bbb S[i])^+\, .
}
\]
Then $\Sigma$ is a homology $(2n+1)$-sphere and the top horizontal map is represented unstably by a fibration
$p\: E \to \Sigma$ whose fiber is a wedge of an even number of spheres. Furthermore ``multiplication by $i$''
induces a fiberwise endomorphism $f\: E\to E$. By construction,
the endomorphism torsion $\tau(E) \in [\Sigma, \Omega K(\Bbb Q(t))] = K_{2n+2}(\Bbb Q(t))$ 
is non-trivial by Theorem \ref{thm:torsion-as-section}.

\appendix
\section{Fundamental theorems of endomorphism $K$-theory} \label{sec:fund-thm}

The following result involves ideas of  
Quillen \cite{Quillen}, Grayson \cite{Grayson} and Waldhausen \cite{Wald_1126}.  We make no claim to originality.
As before, $A$ is an associative ring, $T,S\subset A[t]$ are the multiplicative subsets of centric polynomials of \S\ref{subsec:some-localizations}, and $\widetilde K(\End^S_A)$ the homotopy fiber of the forgetful map $K(\End^S_A) \to K(A)$.

\begin{thm} \label{thm:fund-thm} There are natural homotopy equivalences
\begin{enumerate} 
\item $\Omega K(A[t]_S)\simeq \Omega K(A[t]) \times K(\End^S_A)$, and
\item $\Omega K(A[t^{-1}]_T) \simeq \Omega K(A) \times \widetilde K(\End^S_A)$.
\end{enumerate}
\end{thm}

\begin{rem} We aren't aware of a proof in the literature of the part (1) and so we provide a sketch below. The equivalence is in fact the endomorphism torsion map $\tau(t)$ on the second factor (Theorem \ref{thm:torsion-as-section}). Part (2) is due to Grayson \cite{Grayson}; a non-linear version was also proved by Levikov \cite{Levikov}. It appears the equivalence in (2) is the zeta function $\zeta(t^{-1})$ on the second factor, but the proof requires a translation of \cite{Grayson} to Waldhausen categories and goes beyond the scope of this paper.
\end{rem}

Combining Theorem \ref{thm:fund-thm}  with the fundamental theorem of algebraic $K$-theory \cite[p.~236]{Grayson_Quillen2} relates the $K$-theory of the two localizations: 

\begin{cor} There is a canonical splitting
\[
\Omega K(A[t]_S) \simeq \Omega K(A[t^{-1}]_T) \times \widetilde K(\text{\rm Nil}_A) \, ,
\]
where $\text{\rm Nil}_A \subset \End_A$ is the full subcategory of  nilpotent endomorphisms. 
\end{cor}

In particular, if $A$ is a regular ring then $\widetilde K(\text{\rm Nil}_A) = 0$ and the other two terms are equivalent.
%

\begin{proof}[Proof of  {\rm (1)}] 
We apply the generic fibration theorem of Waldhausen \cite[thm.~1.64.]{Wald_1126} to the category of perfect $A[t]$-chain complexes. Let the $v$-notion of weak equivalence be quasi-isomorphism, and let the $w$-notion be those maps whose homotopy cofiber is $S$-torsion in the sense of Lemma \ref{lem:torsion_is_torsion}. In particular, a chain complex is $w$-acyclic iff it is $S$-torsion.

By Corollary \ref{cor:perfect_is_perfect}, the $K$-theory of the $w$-acyclic objects is identified with $K(\End_A^S)$. Therefore the sequence in the generic fibration theorem becomes
\[
K(\End^S_A) \to K(A[t]) \to K(A[t];w) \, ,
\]
the first map sending $(P,f)$ to the two-term $A[t]$-chain complex $P[t]@> t-f >> P[t]$. By the additivity theorem this map is null-homotopic, hence we get an equivalence
\[
\Omega K(A[t];w) \simeq K(\End^S_A) \times  \Omega K(A[t])\, .
\]
As $A[t]_S$ is flat over $A[t]$, tensoring is exact, so it induces a map
\[
K(A[t];w) \to K(A[t]_S)\, .
\]
This is well-known to give an equivalence after looping (though not an isomorphism on $\pi_0$), which finishes the proof.
\end{proof}

\section{Endomorphism torsion and boundary map} \label{sec:boundary}

In the previous appendix we recalled a fiber sequence of the form
\[
K(\End^S_A) @>0>> K(A[t]) \to K(A[t]_S) \, .
\]
In this appendix we show the endomorphism torsion splits the boundary map. We use this for nonvanishing results for $\tau$ (e.g. Proposition \ref{prop:non-trivial}), and to show that $\tau$ gives the equivalence in part (1) of the fundamental theorem of endomorphism $K$-theory (Theorem \ref{thm:fund-thm}).

This is not the only situation in which a torsion invariant splits the boundary map of a $K$-theory long exact sequence. In general, the boundary map of such a sequence sends every class represented by a $w$-weak equivalence to its homotopy cofiber, which is $w$-acyclic. However, in most cases this analysis only applies to path components, and the torsion map has indeterminacy. In the case of endomorphism torsion, the splitting occurs at the space level and therefore holds for every group $K_i$.

As in the proof of Theorem \ref{thm:fund-thm}, let $C$ be the category of perfect $A[t]$-chain complexes, let $v$ be the quasi-isomorphisms and let $w$ be the maps whose homotopy cofibers are $S$-torsion in the sense of Lemma \ref{lem:torsion_is_torsion}. Let $C^w \subset C$ be the full subcategory of $w$-acyclic objects. Waldhausen's generic fibration then fits into a commuting diagram of the form
\[
\xymatrix{
	\Omega K(C,w) \ar[d]^-\sim \ar[r]^-\partial & K(C^w,v) \ar[r] & K(C,v) \ar[r] & K(C,w) \ar[d] \\
	\Omega K(A[t]_S) & K(\End^S_A) \ar[u]^-\sim \ar[r]^-0 & K(A[t]) \ar@{=}[u] \ar[r] & K(A[t]_S)\, .
}
\]
The second vertical map sends $(P,f)$ to the two-term complex $P[t]@> t-f >> P[t]$, while the outside vertical maps extend scalars along $A[t] \to A[t]_S$.

We observe that the endomorphism torsion naturally lifts along this extension of scalars.  The lifted torsion map comes from the categorical operation that sends $(P,f)$ to the one-term complex $P[t]$ with self-$w$-equivalence $t-f$, instead of the complex $P[t]_S$ with self-isomorphism $t-f$. Composing this with coassembly and then projecting to based loops gives a map
\[ \tau\colon K(\End^S_A) \to LK(C,w) \to \Omega K(C,w) \]
that agrees with the endomorphism torsion along the equivalence
\[ \Omega K(C,w) @>\sim>> \Omega K(A[t]_S)\, . \]
%

\begin{thm} \label{thm:torsion-as-section} The following commutes in the homotopy category
\[
\xymatrix{
	\Omega K(C,w) \ar[r]^-\partial & K(C^w,v) \\
	& K(\End^S_A) \ar[u]_-\sim \ar[ul]^-\tau\, ,
}
\]
i.e., the endomorphism torsion map is a section to the boundary map.
\end{thm}


\begin{proof}
We first recall from \cite{Wald_1126} the definition of the boundary map. Let $\overline wC$ denote the subcategory of maps that are both cofibrations and weak equivalences. Let $\overline w.C$ denote the simplicial category that at simplicial level $k$ is length $k$ flags of maps in $\overline w$,
\[
B_\bullet := B_0 \overset\sim\rightarrowtail B_1 \overset\sim\rightarrowtail \cdots \overset\sim\rightarrowtail B_k\, .
\]
Let $F.(C,C^w)$ denote the same construction but equipped additionally with choices of quotient $B_j/B_i$ for all $i \leq j$ (see \cite[p.344]{Wald_1126} for details). This forms a simplicial Waldhausen category. The forgetful maps $F_k(C,C^w) \to \overline w_kC$ are equivalences of categories, hence equivalences on the $v.$-nerve. Together with the equivalence $v\overline w.C \simeq vw.C$ and the swallowing lemma, this gives an equivalence of spaces
\[ |v.F.(C,C^w)| @>\sim>> |v.\overline w.C| @>\sim>\sigma> |w.C|\ . \]
Here the map $\sigma$ takes a square grid of objects $(A_{ij})$ and maps to the flag of diagonal objects $(A_{ii})$ and the maps between them. Equivalently, it is the most obvious map from the double realization of $v.\overline w.C$ to the single realization of $w.C$, which for each $p \times q$ grid of morphisms, subdivides the corresponding copy of $\Delta^p \times \Delta^q$ and maps each of the simplices to the evident corresponding simplex in $|w.C|$.

We have a commutative square
\[
\xymatrix{
	F.(C,C^w) \ar[d]\ar[r]^-{\partial} &  S.C^w\ar[d]
	\\
	F.(C,C) \ar[r] & S.C
}
\]
where the horizontal maps are induced by the operation
\[
B_\bullet \mapsto B_\bullet/B_0\, .
\]
After passing to $S.$ constructions, the square is homotopy cartesian \cite[cor.~1.56.]{Wald_1126} and $vS.F.(C,C)$ is contractible. Combining this with the generic fibration theorem, we have two homotopy pullback squares
\[
\xymatrix{
	|w.S.C| \ar@{<-}[r]^-\sim & |v.S.F.(C,C^w)| \ar[d]\ar[r]^-{\partial} & |v.S.S.C^w| \ar[d]\ar[r] & |w.S.S.C^w| \ar[d] \\
	& |v.S.F.(C,C)| \ar[r] & |v.S.S.C| \ar[r] & |w.S.S.C|\, .
}
\]
The bottom-left and top-right terms are contractible, giving a four-term homotopy fiber sequence
\[
(|w.S.C| \simeq |vS.F.(C,C^w)|) @>\partial >> |vS.S.C^w| \to |vS.S.C| \to |wS.S.C|\, .
\]
In particular, the map from $|w.S.C|$ to $|v.S.S.C^w|$ may be identified with a (one-fold delooping of) the boundary map $\Omega |wS.C| \to |vS.C^w|$ in the fiber sequence
\[
|vS.C^w| \to |vS.C| \to |wS.C|\, .
\]

We need to check that the identification $|w.S.C| \simeq \Omega |w.S.S.C|$ from the above four-term fiber sequence agrees with the usual one. This follows from the commutativity of the following diagram, because the commuting square at the very bottom gives the usual map $|w.C| \to \Omega |w.S.C|$ by \cite[Lemma 1.5.2]{Wald_1126}, and the composite along the left-hand side agrees in the homotopy category with the identification $|w.S.C| \simeq |vS.F.(C,C^w)|$ from above. For simplicity we have suppressed one copy of $S.$ on every term.
\[ \xymatrix @R=.5em {
	v.F.(C,C^w) \ar[rr] \ar[rd] \ar[dd]_-\sim && w.S.C^w \ar@{=}[dd] \ar[dr] \\
	& v.F.(C,C) \ar[dd]_-(0.75)\sim \ar[rr] && w.S.C \ar@{=}[dd] \\
	w.F.(C,C^w) \ar[rr] \ar[rd] \ar@{<-}[dd]_-\sim && w.S.C^w \ar@{<-}[dd]_-(0.75)\sim \ar[dr] \\
	& w.F.(C,C) \ar@{=}[dd] \ar[rr] && w.S.C \ar@{=}[dd] \\
	w.C \ar[dr] \ar[rr] && {*} \ar[dr] \\
	& w.F.(C,C) \ar[rr] && w.S.C
} \]

Finally, to prove that $\partial \circ \tau$ is the desired equivalence, we extend the domain of $\partial$ by adding on the zero map $K(C,w) \to K(C^w,v)$, giving a map out of the free loop space as indicated below.
\[
\xymatrix{
	L K(C,w) \ar[r]^-\partial & K(C^w,v) \\
	& K(\End^S_A) \ar[u]_-\sim \ar[ul]^-\tau
}
\]
It suffices to prove that this modified diagram commutes.
To prove this we rewrite the modified boundary map as the topmost route in the following diagram.
\[
\xymatrix{
	L|w.S.C| \ar@{<-}[r]^-\sim & L|v.S.F.(C,C^w)| \ar[r]^-{\partial} & L|v.S.S.C^w| \ar@/^1em/[r] & \ar[l] \Omega |v.S.S.C^w| & \ar[l]_-{\sim} {|vS.C^w|} \\
	& \ar[ul]^-\tau \ar@{-->}[u] {|i.S.\End^S_A|} \ar[urrr]_-\sim
}
\]
Then we define a dotted map making both triangular regions in the diagram commute up to homotopy. It arises from an exact functor
\[ \End^S_{A} \to F_1(C,C^w) \]
with the following description. Given $(P,f)$ we regard $P[t] @>t-f>> P[t]$ as a map of one-term $A[t]$-complexes. We take its reverse mapping cone $T^{-}(t-f)$ as in \S\ref{sec:chain_complexes}, though in contrast to that section we use the existing $t$-action, not the action through $\sh^{-1}$. This makes $T^{-}(t-f)$ a two-term complex of $A[t]$-modules, with an acyclic cofibration $\sh^{-1}$ whose quotient is the two-term complex $P[t] @>t-f>> P[t]$. This describes the desired exact functor to $F_1(C,C^w)$:
\[ (P,f) \mapsto \left(T^{-}(t-f) @>\sh^{-1}>> T^{-}(t-f), \ T^{-}(t-f) / \sh^{-1} \cong (P[t] @>t-f>> P[t]) \right)\ . \]
It produces a map from $\End_A^S$ to simplicial level 1 of $F.(C,C^w)$, with the property that the two faces give the same map to $F_0(C,C^w)$, hence on realizations it gives a map to the free loop space.
%

Now we check the two triangular regions commute. The one on the right is fairly straightforward: both routes arise in the same way from the exact functor
\[ \End^S_A \to S_1 C^w \cong C^w \]
taking $(P,f)$ to the complex $P[t] @>t-f>> P[t]$. For the region on the left, we recall that the weak equivalence is a composite of the two maps
\[
\xymatrix{
	L|w.S.C| \ar@{<-}[r]^-\simeq_-\sigma & L|w.w.S.C| \ar@{<-}[r]^-\simeq & L|v.S.F.(C,C^w)|\, .
}
\]
Unraveling the definition of $\sigma$, the composite arises from the natural transformation of exact functors
\[ \End^S_A \times [1] \to C \]
sending $(P,f)$ to the map of chain complexes
\[ \sh^{-1}\colon T^{-}(t-f) \to T^{-}(t-f)\, . \]
We collapse these mapping cylinders onto their ends to get a simpler map sending $(P,f)$ to the map of chain complexes
\[ P[t] @>t-f>> P[t] . \]
This matches precisely the definition of $\tau$. The collapsing operation modifies our original map by a homotopy, so we conclude the triangular region on the left commutes up to homotopy.
\end{proof}

\section{The sphere theorem for $\End_A$ and $\End_A^S$}\label{sec:sphere_theorem}

In this appendix we recall a construction in homological algebra that allows us to factor each map $X \to Y$ of perfect chain complexes as
\[ X \to X_{m+1} \to X_{m+2} \to \ldots \to X_n @>\sim>> Y \]
where each $X_q/X_{q-1}$ has homology that is finitely generated projective and concentrated in degree $q$. We then generalize this construction so that it applies to $A$-perfect chain complexes over $A[t]$, i.e. to chain complexes of endomorphisms, with or without $S$-torsion. This is used in \S \ref{sec:chain_complexes} to show that the chain complex models for $K(\End_A)$ and $K(\End_A^S)$ are equivalent to the classical models (Proposition \ref{four_models_for_endomorphisms}).

%

Let $A$ be an associative ring. Let $f\colon X \to Y$ be a map of chain complexes of $A$-modules, $P$ a projective $A$-module, and $k\colon P \to H_n(Y,X)$ any $A$-linear map.
\begin{lem}\label{kill_one}
	There is a factorization $X \to X' \to Y$ in which $H_q(X',X)$ is $P$ in degree $n$ and $0$ otherwise, and along this identification the induced map $P \to H_n(Y,X)$ agrees with $k$.
\end{lem}

\begin{proof}
	Recall that $H_*(Y,X) = H_*(C(f))$ where $C(f)$ is the mapping cone. We have $C(f)_n \cong Y_n \oplus X_{n-1}$ with boundary map given by the boundary map of $Y$, the negated boundary of $X$, and $f\colon X_{n-1} \to Y_{n-1}$. Therefore an $n$-cycle consists of an $(n-1)$-cycle $x \in Z_{n-1}(X)$ and an $n$-chain $y \in Y_n$ such that $\partial y = f(x)$. In other words, a map from a ``$(n-1)$-sphere'' into $X$ and an extension to an ``$n$-disc'' in $Y$.
	
	Since $P$ is projective, $k$ lifts to a map $P \to Z_n(C(f))$, giving a commuting diagram
	\[ \xymatrix{
		P \ar[d]^-h \ar[r]^-g & Y_n \ar[d]^-\partial \\
		Z_{n-1}(X) \ar[r]^-f & Z_{n-1}(Y)\ .
	} \]
	Set $X'_n = X_n \oplus P$, with the boundary map on $P$ given by $h$, and $X'_* = X_*$ in all other degrees. It is immediate that this is a chain complex and there is a factorization of $f$ into maps of chain complexes $X \to X' \to Y$, the latter map defined on $P$ by $g$. The relative homology of $(X',X)$ is also clearly $P$, and the isomorphism is split by
	\[ P \to C(X \to X')_n \cong (X_n \oplus P) \oplus X_{n-1} \]
	\[ p \mapsto (0,p,h(p))\ . \]
	(The latter term is needed for this map to land in the cycles, and therefore induce a map to homology.)
	Therefore the resulting map $P \to C(X \to X')_n \to C(f)_n$ sends $p$ to $(g(p),h(p))$ as desired.
\end{proof}

As a result, by the long exact sequence
\[ 0 \to H_{n+1}(Y,X) \to H_{n+1}(Y,X') \to P @>k>> H_n(Y,X) \to H_n(Y,X') \to 0\ , \]
if we replace $X$ by $X'$, the new relative homology $H_*(Y,X')$ agrees with the old relative homology $H_*(Y,X)$ in all degrees except $n$, where it is the cokernel of $k$, and degree $n+1$, where it is an extension of $H_{n+1}(Y,X)$ by the kernel of $k$. In particular, if $k$ is surjective then this procedure kills $H_n(Y,X)$ and only changes $H_{n+1}(Y,X)$ in the process.

We recall two more preliminaries. Let $H(A)$ the category of $A$-modules that admit a finite resolution by finitely-generated projective modules.

\begin{lem}\label{finite_depth_char}
	A module is in $H(A)$ iff it occurs as the sole nonzero homology group of a perfect complex.
\end{lem}

\begin{proof}
	The forward implication is obvious. 
	If our module is the lone homology group of a strictly perfect complex, and the lowest nonzero group of the complex $P_0$ is beneath the lowest homology group, we include a copy of $P_0 \to P_0$ into the bottom of the complex and take the cofiber, giving a new perfect complex that is shorter but has the same homology.
	After performing this move finitely many times, we arrive at the desired resolution for our module.
\end{proof}

\begin{lem}\label{shorten_resolution}\cite[II.7.7.1]{Weibel_K-book}
	If $H$ has a length $(n+1)$ resolution by finitely generated projective $A$-modules and $0 \to K \to P \to H \to 0$ is short exact with $P$ finitely generated projective, then $K$ has a resolution of finitely generated projective $A$-modules of length $\max(0,n)$.
\end{lem}


Now we may verify the main hypothesis of the sphere theorem. Recall that a map of chain complexes $X \to Y$ is $m$-connected if $H_q(Y,X) = 0$ for $q \leq m$.
\begin{prop}\label{sphere_hyp}
	Given an $m$-connected map $X \to Y$ of $A$-perfect $A$-chain complexes, there is a factorization
	\[ X \to X_{m+1} \to X_{m+2} \to \ldots \to X_n @>\sim>> Y \]
	where each $X_q/X_{q-1}$ has homology that is finitely generated projective over $A$ and concentrated in degree $q$.
\end{prop}

\begin{proof}
	The mapping cone has lowest homology in degree $(m+1)$. As the mapping cone is perfect, its lowest homology group is finitely generated, therefore we have a surjective map $P \to H_n(Y,X)$ where $P$ is finitely generated free. The procedure of Lemma \ref{kill_one} produces $X_{m+1}$ such that $X_{m+1}/X$ has homology $P$ concentrated in degree $(m+1)$ and $X_{m+1} \to Y$ is $(m+1)$-connected. We repeat until we have exhausted all of the relative homology groups of the original map. Then $C(f)$ has only one homology group remaining, so it lies in $H(A)$ by Lemma \ref{finite_depth_char}. By Lemma \ref{shorten_resolution}, if we continue the procedure then after finitely many steps the homology group will be finitely generated projective, in which case we apply Lemma \ref{kill_one} once more with $P$ as the last remaining homology group.
\end{proof}

%

Now we modify this argument to work for endomorphisms, and endomorphisms with torsion. In each case, an additional trick is needed.
\begin{lem}\label{kill_one_depth_one}
	Lemma \ref{kill_one} remains true if, instead of being projective, $P$ has a length-one resolution by projective modules
	\[ 0 \to P_1 @>i>> P_0 @>j>> P \to 0\ . \]
\end{lem}
\begin{proof}
	We lift the map $P_0 @>j>> P @>k>> H_n(Y,X)$ to $P_0 \to C(f)_n$ and apply the procedure of Lemma \ref{kill_one} to this map. This gives the lower square below. We then add $P_1$ as shown and define the map $b$ to $Y_{n+1}$ by noting that the composite to $Y_n$ is zero on homology, therefore must land in the boundaries, and therefore has a lift because $P_1$ is projective. The total modified complex and its map to $Y$ now look like
	\[ \xymatrix{
		X_{n+1} \oplus P_1 \ar[d]_-{\partial \oplus i} \ar[r]^-{(f,b)} & Y_{n+1} \ar[d]^-\partial \\
		X_n \oplus P_0 \ar[d]_-{(\partial,h)} \ar[r]^-{(f,g)} & Y_n \ar[d]^-\partial \\
		X_{n-1} \ar[r]^-f & Y_{n-1}
	} \]
	This gives a factorization $X \to X'' \to Y$ of $A$-perfect $A[t]$-chain complexes where $H_*(X'',X)$ is $P$ concentrated in degree $n$, and the resulting map to $H_n(Y,X)$ is $k$. So we again get the same conclusions as in Lemma \ref{kill_one} for the effect of this operation on relative homology.
\end{proof}
\begin{rem}
	This argument breaks down if the resolution of $P$ is longer, because one begins to encounter obstructions in the homology groups of $Y$.
\end{rem}
%
%
\begin{prop}\label{sphere_hyp_end}
	Given an $m$-connected map $X \to Y$ of $A$-perfect $A[t]$-chain complexes, there is a factorization in $A[t]$-chain complexes
	\[ X \to X_{m+1} \to X_{m+2} \to \ldots \to X_n @>\sim>> Y \]
	where each $X_q/X_{q-1}$ has homology that is finitely generated projective over $A$ and concentrated in degree $q$. (Therefore every term of the factorization is $A$-perfect.)
\end{prop}

\begin{proof}
	The mapping cone has lowest homology in degree $(m+1)$, and is finitely generated over $A$. Therefore there is a surjective $A$-linear map $k\colon P \to H_{m+1}(Y,X)$ where $P$ is a finitely generated projective $A$-module (with no $A[t]$-action).
	
	The key observation is that we can always endow such a $P$ with an $A[t]$-action such that $k$ is $A[t]$-linear. Without loss of generality $P$ is free, and then we define the action one basis element at a time, by applying $k$, applying the $t$ action in $H_{m+1}(Y,X)$, then taking a lift along $k$. We then extend to the rest of $P$ to form an $A$-linear map $\alpha\colon P \to P$ such that $k\alpha = tk$. This makes $P$ into an $A[t]$-module such that $k$ is $A[t]$-linear.
	
	Then by the characteristic sequence, $P$ has a length-one resolution by projective $A[t]$-modules. By Lemma \ref{kill_one_depth_one}, we can therefore form the desired factorization through $X_{m+1}$. The rest of the proof now proceeds just as in Proposition \ref{sphere_hyp}.
%
\end{proof}

\begin{prop}\label{sphere_hyp_tors}
	Given an $m$-connected map $X \to Y$ of $A$-perfect, $S$-torsion $A[t]$-chain complexes, there is a factorization in $A[t]$-chain complexes
	\[ X \to X_{m+1} \to X_{m+2} \to \ldots \to X_n @>\sim>> Y \]
	where each $X_q/X_{q-1}$ has homology that is $S$-torsion, finitely generated projective over $A$ and concentrated in degree $q$. (Therefore every term of the factorization is $A$-perfect and $S$-torsion.)
\end{prop}

\begin{proof}
	Since $H_{m+1}(Y,X)$ is $A[t]$-finitely generated and $S$-torsion, there is a single element $p(t) \in S$ such that $p(t)$ acts by zero on $H_{m+1}(Y,X)$. Then we get a surjective $A[t]$-linear map $(A[t]/p(t))^{\oplus i} \to H_{m+1}(Y,X)$. Observe that $(A[t]/p(t))^{\oplus i} \cong A^{\oplus i(\deg p)}$ is finitely generated free as an $A$-module, and $S$-torsion as an $A[t]$-module, so we may kill it as before using Lemma \ref{kill_one_depth_one} without leaving the category of perfect $S$-torsion complexes. We continue this trick for the remaining steps, except the final step where the last remaining homology group is finitely generated projective, in which case (just as in Proposition \ref{sphere_hyp}) we take $P$ to be that last remaining homology group.
%
\end{proof}
\bibliography{john}



\end{document}